\documentclass[BCOR=12mm, DIV=default,twoside, pagesize, draft]{scrartcl}

\addtokomafont{caption}{\small}
\setkomafont{captionlabel}{\sffamily\bfseries}
 
\usepackage{tikz}  
\usetikzlibrary{arrows,automata}

\usepackage{amsmath, amssymb, amsthm}
\usepackage[utf8x]{inputenc}
\usepackage{libertine}
\usepackage[T1]{fontenc}
\usepackage[english]{babel}
\usepackage{url}  
\usepackage{faktor,csquotes,color}
\usepackage{pgfplots} 
\usepackage{enumitem}

\usepackage[final]{changes}

\definechangesauthor[name={Soeren}, color=red]{soe}
\definechangesauthor[name={Kristoffer}, color=blue]{kri}
\newcommand{\stkout}[1]{\ifmmode\text{\sout{\ensuremath{#1}}}\else\sout{#1}\fi}
\setdeletedmarkup{\stkout{#1}}

\usepackage{units,varioref}

\usepackage[shortcuts]{extdash}
\usepackage{setspace,scrpage2}
\usepackage{mathtools}
\mathtoolsset{showonlyrefs,showmanualtags}
\numberwithin{equation}{section}

\allowdisplaybreaks 
\clearscrheadfoot
\automark[section]{section}
\ohead{\headmark}
\ofoot[\pagemark]{\pagemark}

\newtheoremstyle{break}{\topsep}{\topsep}{\itshape}{}{\bfseries}{.}{\newline}{}
\newtheoremstyle{exampl}{\topsep}{\topsep}{\upshape}{}{\bfseries}{.}{\newline}{}

\theoremstyle{plain}
\newtheorem{thm}{Theorem}[section]
\newtheorem{lem}[thm]{Lemma}  
\newtheorem{prop}[thm]{Proposition}  
 
\newtheorem{ass}[thm]{Assumption} 

\theoremstyle{break}

\theoremstyle{definition}
\newtheorem{defi}[thm]{Definition}

\newtheorem{ex}[thm]{Example}

\theoremstyle{exampl}

\theoremstyle{remark}
\newtheorem{rem}[thm]{Remark}

\def\state{E}

\DeclareMathOperator{\E}{\mathbb E}

\definecolor{mygray}{gray}{.5}

\subject{\begin{flushleft}
{\normalfont 
{\footnotesize Accepted for publication in:} \textcolor{blue}{\normalsize Stochastic Processes and their Applications, 2020+.}\\*
\vspace{-2mm}{\normalsize DOI:}\textcolor{blue}{\normalsize 10.1016/j.spa.2019.08.010} 
}\end{flushleft} 
}

\title{On time-inconsistent stopping problems and mixed strategy stopping times}
 
\author{S\"oren Christensen\footnote{Department of Mathematics, Kiel University, Germany. E-mail address: christensen@math.uni-kiel.de.} \and Kristoffer Lindensj\"o\footnote{Department of Mathematics, Stockholm University, Sweden. E-mail address: kristoffer.lindensjo@math.su.se.}}

\begin{document}

\maketitle

\begin{abstract}
A game-theoretic framework for time-inconsistent stopping problems where the time-inconsistency is due to the consideration of a non-linear function of an expected reward is developed. 
A class of mixed strategy stopping times that allows the agents in the game to 
jointly choose the intensity function of a Cox process is introduced and motivated. 
A subgame perfect Nash equilibrium is defined.
The equilibrium is characterized and other necessary and sufficient equilibrium conditions including a smooth fit result are proved. 
Existence and uniqueness are investigated.
A mean-variance and a variance problem are studied. The state process is a general one-dimensional It\^{o} diffusion. 
\end{abstract}

\noindent \textbf{Keywords:} Conditional Poisson process, Cox process, Equilibrium stopping time, Mean-variance criterion, Mixed strategies, Optimal stopping, Subgame perfect Nash equilibrium, Time-inconsistency, Variance criterion.\\

\noindent \textbf{AMS MSC2010:} 
60G40; 60J70; 91A10; 91A25; 91G80; 91B02; 91B51.

\section{Introduction} \label{intro}
Consider a diffusion $X$ and the classical problem of choosing a stopping time $\tau$ that maximizes
\[
\E_x(h(X_\tau)),
\]
where $h$ is a nice deterministic function. Recall that the solution to this problem is consistent in the sense that the optimal rule for stopping, i.e. 'stop the first time that $X$ enters the stopping region', is independent of the initial state $x$. 
In this paper, we generalize this setting by considering non-linear nice deterministic functions $f$ and $g$ and the problem of choosing a stopping time $\tau$ that maximizes
\begin{align}
 \E_x(f(X_\tau))+g(\E_x(h(X_\tau))).
\end{align}
The optimal stopping rule for this problem will, in contrast, typically depend on the initial state $x$, which means that it will not generally satisfy Bellman's principle of optimality. In the literature this is known as \emph{time-inconsistency}.
Generalizations of classical stochastic control problems leading to time-inconsistency are discussed in Section \ref{reasons-time-incon} below.

Time-inconsistent problems are typically studied using one of the following approaches: 

\begin{itemize}
\item The \emph{game-theoretic approach} studied in the present paper, which means formulating the problem as a game and look for equilibrium stopping times, see Definition \ref{def:mix_strat} and Section \ref{equilibrium-disc} below.

\item  The \emph{pre-commitment approach}, which means formulating the problem for a fixed initial state and allowing the corresponding optimal stopping rule to depend on that initial state.

\item The \emph{dynamic optimality approach}, developed in \cite{pedersen2016optimal}. See also \cite{christensen2017finding} for a short description.
\end{itemize}

In \cite{christensen2017finding} we developed a game-theoretic framework for time-inconsistent stopping problems covering  endogenous habit formation and non-exponential discounting. In the present paper, a game-theoretic framework for time-inconsistent stopping problems that can handle e.g. mean-variance problems is developed. See Section \ref{reasons-time-incon} for an explanation of these terms. 
In the present paper we also define mixed strategy stopping times by allowing the agents in the game to jointly choose the intensity function of a Cox process that is used as a randomization device for the stopping decision, see Definition \ref{def:mix_strat} and the motivation in Section \ref{equilibrium-disc}. This type of mixed strategy stopping time appears to be novel, although other types of mixed strategies in stopping games have been considered and the Cox process has been used in other ways in different kinds of stopping games, see Section \ref{prev-lit}. 

The rest of the paper is organized as follows: In Section \ref{sec:problem-formulation} we formulate the time-inconsistent stopping problem in detail and give the definitions of mixed strategy stopping time and equilibrium. These definitions are motivated and discussed in Section \ref{equilibrium-disc}.
In Section \ref{sec:equilibrium-conditions} the equilibrium is characterized and other results with necessary and sufficient conditions for equilibrium are proven, these are the main results of the present paper, see
Theorem \ref{main-thm-imm-stp},  
Theorem \ref{smoothfit-thm}, 
Theorem \ref{sufficient-cond-thm} and
Theorem \ref{nec-cond-THM-for-C}. 
In Section \ref{applications} the developed theory is applied. 
Section \ref{applications:var} studies a variance problem  and underlies the necessity of using mixed strategy equilibria. 
Section \ref{applications:mean-var} studies a mean-variance problem. In particular, Theorem \ref{mean-varthm} shows that the mean-variance problem has no equilibrium for some parameter specifications, implying that we cannot generally expect equilibria to exist.
Section \ref{subsec:example-2-equi} studies an example with two different equilibria, implying that we cannot in general expect equilibrium uniqueness.

\subsection{Previous literature} \label{prev-lit}
 The game-theoretic approach to time-inconsistent problems was first considered by Strotz \cite{strotz} in a seminal paper studying dynamic utility maximization under non-exponential discounting. The approach was further developed by 
Selten \cite{selten1965spieltheoretische,selten1975reexamination}  
who introduced the notion of subgame perfect Nash equilibrium, which is a refinement of the Nash equilibrium suitable for dynamic games. 

Recently, there has been a substantial effort to develop the literature on the game-theoretic approach to time-inconsistent control problems. 
The main theoretical result of time-inconsistent Markovian stochastic control is a characterization of an equilibrium as a solution to a generalized HJB equation called the extended HJB system, see \cite{tomas-continFORTH,tomas-disc,lindensjo2017timeinconHJB}. 
Recently a considerable literature using the extended HJB system to study time-inconsistent control problems has emerged, examples include \cite{bensoussan2014time,bjork2014mean,he2013optimal,kronborg2015inconsistent,li2013optimal}.

The development of the literature on the game-theoretic approach to time-inconsistent stopping problems is in an earlier stage. Recent papers include \cite{bayraktar2018,
christensen2017finding,
Duraj2017Optimal,
huang2018time,
huang2017stopping,
huang2017optimal, 
huang2017optimalDISC}. Section \ref{reasons-time-incon} describes references studying particular time-inconsistent stopping problems, while Section \ref{equilibrium-disc} contains a further review focusing on the choice of definition for pure and mixed strategies and equilibria. Section \ref{applications} contains further references to papers studying mean-variance and variance problems. For short surveys of time-inconsistent stopping problems we also refer to \cite{bayraktar2018,christensen2017finding,pedersen2016optimal}.

\deleted[id=kri,remark={This paragraph has been moved to Section 2.1 in order to coherently consider a comment of a referee.}]{In \cite{touzi2002continuous} a continous-time Dynkin game with mixed strategies defined as randomized stopping times is studied. It is instructive to note that the number of players in the game of this paper is finite while the number of players in the game of the present paper is uncountable; in the framework of the present paper it is the Cox process construction of mixed stopping strategies that makes it possible to identify mixed equilibrium strategies.}

Recent papers on time-inconsistent stopping problems and the dynamic optimality and pre-commitment approaches include \cite{miller2017nonlinear,pedersen2016optimal}.

In \cite{nutz2018mean}, a mean-field optimal stopping game for e.g. bank-runs is studied. The default time is modeled as the first jump time of a given Cox process. In \cite{dupuis2002optimal,guo2005stopping}, optimal stopping problems where stopping can only occur at exogenously determined Poisson jump times are studied.

\subsection{Reasons for time-inconsistency}\label{reasons-time-incon}
In this section we formulate simple examples to give an idea about the type of time-inconsistent problems that are studied in finance and economics. In particular, we see that problems of mean-variance type can be studied in the framework of the present paper, whereas problems of endogenous habit formation and non-exponential discounting type can be studied in the general framework of \cite{christensen2017finding}.

\textbf{Mean-variance optimization/utility}: Suppose the process $X$ corresponds to the price of an asset that an investor wants to sell. A plan for the asset sale is formalized as a stopping time. The utility of the investor, given the current price $x$ and an asset sale plan $\tau$, is, 
\[\E_x(X_\tau) - \gamma \mbox{Var}_x(X_\tau), \mbox{ with $\gamma>0$.}\] 
The economic interpretation is that the there is a tradeoff between the expected selling price and risk measured by variance. The parameter $\gamma$ corresponds to risk aversion. The game-theoretic approach to stopping problems of this type is studied in Section \ref{applications:mean-var} and \cite{bayraktar2018}.

\textbf{Endogenous habit formation}: Consider the asset selling problem but suppose the investor for a given current price $x$ considers, 
\begin{align} \label{Endogenous-problem}
\E_x(F(X_\tau,x)),
\end{align}
where $F(\cdot,x)$ is a standard utility function for each fixed $x$. The economic interpretation is that the investor  dynamically updates a habitual preference for the expected utility of the selling price that is based on the current asset price. The game-theoretic approach to stopping problems of this type is studied in \cite{christensen2017finding}.

\textbf{Non-exponential discounting}: Consider the asset selling problem but suppose the investor at time $t$ given the current price $x$ considers, 
\[\E_{t,x}(\delta(\tau-t)F(X_\tau))\added[id=kri,remark={}]{,}\deleted[id=kri,remark={}]{.}\]
where $F$ is a standard utility function and $\delta$ is a discounting function 
--- that is, $\delta:[0,\infty)\rightarrow [0,1]$ is non-increasing with $\delta(0)=1$  --- 
which cannot be written as an exponential discounting function. The game-theoretic approach to stopping problems of this type is studied in \cite{huang2018time,
huang2017stopping,
huang2017optimal,
huang2017optimalDISC}. Problems of this type can be studied in the general framework of \cite{christensen2017finding} by considering the time-space process.

\section{Problem formulation} \label{sec:problem-formulation}
Let $(\Omega,\mathcal{F},(\mathcal{F}_t)_{t\geq 0},\mathbb{P}_x)$ be a filtered probability space carrying a one-dimensional Wiener process $W$. Let $X$ be a one-dimensional diffusion living on an open interval $\state =(\alpha,\beta)$, where $-\infty\leq\alpha\leq\beta\leq\infty$, which is the unique strong solution to the SDE
\begin{align} \label{the-diffusion}
dX_t = \mu(X_t)dt + \sigma(X_t)dW_t, \enskip X_0 = x.
\end{align}
The coefficients $\mu:E\rightarrow\mathbb{R}$ and $\sigma:E\rightarrow (0,\infty)$ are continuous and satisfy conditions guaranteeing the existence of a unique strong solution, see e.g. \cite{karatzas2012brownian}. Moreover, for each continuous function $\lambda:E\rightarrow [0,\infty)$ the filtered probability space is assumed to carry an $X$-associated Cox process denoted by $N^\lambda$, meaning that $N^\lambda$ is a Poisson process with intensity corresponding to $\lambda(X_t)$ conditional on the natural filtration generated by $X$, see e.g. \cite[Sec. 6.6]{bielecki2013credit}. 
It is assumed that the filtration $(\mathcal{F}_t)_{t\geq 0}$ satisfies the usual conditions 
and that $x \mapsto  \mathbb{P}_x(F)$ is measurable for each $F \in \mathcal{F}$. The associated expectations are denoted by $\mathbb{E}_x$. 
It is assumed that a measurable time shift operator $\theta$ with 
$X_\tau \circ \theta_{\tau_h} = X_{\tau \circ \theta_{\tau_h}  + \tau_h}$ exists, where $\tau$ is a, possibly infinite, stopping time (with respect to $(\mathcal{F}_t)_{t\geq 0}$) and
\[\tau_h:=\inf\{t\geq 0:|X_t-X_0|\geq h\}.\]
Now consider the functions $f,h:E \rightarrow \mathbb{R}$ and $g:\mathbb{R} \rightarrow \mathbb{R}$ satisfying Assumption \ref{ass:general-lemmas}  (below) and the problem of finding a stopping time $\tau$ that maximizes
\begin{align} \label{the-problem}
J_{\tau}(x): = \E_x(f(X_\tau)) + g(\E_x(h(X_\tau))). 
\end{align}
\begin{rem} \label{convention-tau-infty} We use the convention that $h(X_\tau):= \limsup_{t\rightarrow\infty}h(X_t)$ on $\{\tau = \infty\}$ and similarly for $f$. We assume that the limits $g(\infty):= \lim_{x\rightarrow \infty}g(x)$ and $g(-\infty):= \lim_{x\rightarrow -\infty}g(x)$  exist, see Assumption \ref{ass:general-lemmas}.
\end{rem}

Let us   specify which type of stopping times are admissible (Definition \ref{def:mix_strat}) and then give the equilibrium definition (Definition \ref{def:equ_stop_time}). For a fixed stopping time $\tau$ we define the functions $\phi_\tau$ and $\psi_\tau$ by,
\[\phi_\tau(x)=\E_x(f(X_\tau)) \mbox{ and } \psi_\tau(x)=\E_x(h(X_\tau)).\]

\begin{defi} \label{def:mix_strat} Consider a continuous function $\lambda:E \rightarrow [0,\infty)$ and the corresponding Cox process $N^\lambda$. Let $\tau^\lambda :=  \inf\{t\geq 0: N^\lambda_t\neq N^\lambda_{t-}\}$. Let $C \subset E$ be an open set and let 
$\tau^C :=  \inf\{t\geq 0: X_t \notin C\}$. 
Then $\tau^{\lambda,C} :=  \tau^\lambda\wedge \tau^C$ is said to be a mixed 
(Markov) strategy stopping time.  A mixed strategy stopping time $\tau^{\lambda,C}$ is said to be admissible if the function $J_{\tau^{\lambda,C}}$ in \eqref{the-problem} is well-defined and the functions $\phi_{\tau^{\lambda,C}}$ and $\psi_{\tau^{\lambda,C}}$ are continuous. %
The space  of admissible mixed strategy stopping times is denoted by $\mathcal{N}$.
\end{defi}
Usually we write $\phi_{\lambda,C}$ instead of $\phi_{\tau^{\lambda,C}}$ and similarly for $\psi_{\tau^{\lambda,C}}$ and $J_{\tau^{\lambda,C}}$. We remark that the requirement that $\phi_{\lambda,C}$ and $\psi_{\lambda,C}$ must be continuous in order for $\tau^{\lambda,C}$ to be admissible is a technical condition. For $\tau^{\lambda,C},\tau^{\eta,D}\in\mathcal{N}$ we will use the notation
\begin{align} 
\tau^{\lambda,C} \diamond \tau^{\eta,D}(h)= I_{\{\tau^{\eta,D} \leq \tau_h\}}\tau^{\eta,D} + I_{\{\tau^{\eta,D} > \tau_h\}}({\tau^{\lambda,C}}\circ \theta_{\tau_h}+\tau_h).
\end{align}

\begin{defi}\label{def:equ_stop_time} A stopping time $\hat\tau  \in\mathcal{N}$ is said to be a (mixed Markov strategy) equilibrium stopping time if the equilibrium condition
\begin{align}
\liminf_{h\searrow 0}\frac{J_{\hat\tau}(x)-J_{\hat\tau\diamond \tau^{\eta,D}(h)}(x)}{\E_x(\tau_h)}&\geq 0\label{eqdef2}
\end{align} 
is satisfied for each $\tau^{\eta,D}\in\mathcal{N}$ and each $x\in\state$. If $\hat\tau$ is an equilibrium stopping time then $J_{\hat\tau}(x),x\in\state$, is said to be the corresponding equilibrium value function.
\end{defi}
 For a motivation of these definitions see Section \ref{equilibrium-disc}. This paper is devoted to the question of how to find equilibrium stopping times of the type in Definition \ref{def:equ_stop_time}.

We denote the characteristic operator of $X$ by $A_X$, i.e. for any function $f:E \rightarrow \mathbb{R}$,
\begin{align*}
A_Xf(x) = \lim_{h\searrow 0}\frac{\E_x(f(X_{\tau_h}))-f(x)}{\E_x(\tau_h)},
\end{align*}
whenever this expression exists. Recall that if $f\in\mathcal{C}^2(E)$ then
\begin{align*}
A_Xf(x)  = \mu(x)f'(x) +  \frac{1}{2}\sigma^2(x)f''(x).
\end{align*}
Throughout the paper we assume that the functions $f,g$ and $h$ in \eqref{the-problem} satisfy the following conditions: 
\begin{ass} \label{ass:general-lemmas} \quad
\begin{itemize}
\item $f,h\in \mathcal{C}^2(E)$ and $g\in \mathcal{C}^3(\mathbb{R})$.
\item $g(\infty)$ and $g(-\infty)$ exist in $[-\infty,\infty]$.
\item  $f$ is either bounded from below or above on $\state$. This also holds for $h$.
\end{itemize} 
\end{ass}

\subsection{Motivation and discussion of the definitions of mixed strategy stopping time and equilibrium}\label{equilibrium-disc}
We remark that this section is only of motivational value. Let us first describe how to interpret the time-inconsistent stopping problem \eqref{the-problem} as an intrapersonal game. The non-linearity in \eqref{the-problem} implies that Bellman's optimality principle does not generally hold which means that if a stopping rule, e.g. 'stop the first time $X$ exits $C$', is optimal (in the usual sense) given the starting value $x$, then that stopping rule will generally not be optimal given another starting value $y$  --- note that this is easiest to see for the problem \eqref{Endogenous-problem} where the payoff depends directly on the current state $x$. Based on this, the game-theoretic approach is to view \eqref{the-problem} as a stopping problem for a person who decides when to stop the process $X$ but whose preferences change as the current state $x$ changes. This person is viewed as comprising different versions of herself, one version for each $x$, and these $x$-versions are viewed as agents who play an \emph{intrapersonal} dynamic game against each other regarding when to stop the process $X$. This interpretation is inline with the invention of Strotz and the literature on time-inconsistent control and stopping problems, see Section \ref{prev-lit} and Section \ref{reasons-time-incon}.

\begin{ex} \label{motiv-exa} To clarify\deleted[id=kri,remark={We have shortened this example due to a referee comment.}]{ the game and especially} the notion\deleted[id=kri,remark={}]{s} of pure and mixed strategy stopping 
times\deleted[id=kri,remark={}]{ as defined in the present paper}, 
we here formulate a simple example in discrete 
time\deleted[id=kri,remark={}]{ (which can be embedded into a continous time setting)}\added[id=kri,remark={}]{, in line with the definitions of \cite{bayraktar2018} (cf. the definition of time-homogeneous randomized stopping time in \cite[Section 2]{bayraktar2018})}. 
Suppose $X$ is a \added[id=kri,remark={}]{discrete time} Markov chain\added[id=kri,remark={}]{ living }\deleted[id=kri,remark={}]{that lives}on $\{1,2,3\}$ and consider a variance problem, i.e. suppose 
$J_{\tau}(x) = \mbox{Var}_x(X_\tau)$. 
In this game there are three $x$-agents, $x=1,2,3$\deleted[id=kri,remark={}]{. The} \added[id=kri,remark={}]{whose }
potential \added[id=kri,remark={}]{individual} actions \deleted[id=kri,remark={}]{of each $x$-agent}are $\{$\emph{stop, continue}$\}$. 
For game theory in general, a pure strategy determines the action of an agent based only on payoff relevant information. 
A pure strategy for an $x$-agent is therefore a decision to stop or continue based on $x$. 
Moreover, a strategy profile describes a fully specified configuration of the strategies of all agents in a game; 
an example \deleted[id=kri,remark={}]{in this case }is
$\{$\emph{$1$-agent stops, $2$-agent stops, $3$-agent continues}$\}$.
It is therefore natural to interpret stopping times of the type $\tau^C$ as pure strategy profiles\deleted[id=kri,remark={}]{, and we call them pure strategy stopping times}. 
For game theory in general, a mixed strategy is a strategy where an agent uses a randomization device to select a pure strategy.\deleted[id=kri,remark={}]{In the present game each $x$-agent has only two possible actions, \emph{stop} or \emph{continue}.} 
Hence, an $x$-agent choosing a mixed strategy corresponds to this $x$-agent choosing a biased coin which is flipped every time the state process is at $x$ and whose outcome when flipped determines whether the $x$-agent stops or continues. An example of a mixed strategy profile is\deleted[id=kri,remark={}]{ therefore}:  
$\{$\emph{$1$-agent stops with probability $0.1$, $2$-agent stops with probability $0.8$, $3$-agent stops with probability $1$}$\}$.
\end{ex}
\deleted[id=kri,remark={}]{The natural continous time interpretation of stopping according to sequential coin-flipping with coins whose bias depends on the current state $x$ is to stop at the first jumping time of an $X$-associated Cox process. 
With Example \ref{motiv-exa} in mind we thus interpret stopping times of the type $\tau^{\lambda,C}$ as mixed strategy profiles for time-inconsistent stopping problems, and call them mixed strategy stopping times.}

\added[id=kri,remark={}]{In continous time we interpret the stopping time $\tau^{\lambda,C}$ (Definition \ref{def:mix_strat}) as a mixed strategy profile for our time-inconsistent stopping time, which we will now motivate. 
The interpretations of the discrete and continous time mixed strategies, between which there are, as we shall see, differences, are as follows.  
In discrete time (cf. Example \ref{motiv-exa}), the interpretation is that if we at a time $t-1$ have not stopped and at time $t$ observe some state $x\in C$ then we flip a coin with a bias that depends on $x$ and stop at $t$ if the outcome is, say, heads. 
In continous time, the interpretation is that if we at time $t$ have not stopped and observe some state $x\in C$ then we stop during $(t,t+dt]$ with probability $\lambda(x)dt$; note that this interpretation relies on $\lambda$ being continuous, cf. Definition \ref{def:mix_strat}. 
Hence, although there are differences between the discrete and continous time definitions we see that, intuitively, also for the continous time definition holds that the $x$-agents use randomization in order to determine whether to stop or not, and in this sense it is appropriate to interpret $\tau^{\lambda,C}$ as a mixed strategy profile for our continous time time-inconsistent stopping problem.}

Let us now motivate our choice of equilibrium condition \eqref{eqdef2}\added[id=kri,remark={}]{, which is an adaptation of the classical equilibrium definition of time-inconsistent stochastic control, see Remark \ref{usual-eq} below.} 
In general, a subgame perfect Nash equilibrium is for a dynamic game a strategy profile that forms a Nash equilibrium at each point in time. The exact mathematical definition of a subgame perfect Nash equilibrium is to some extent a matter of choice and should be chosen in order to obtain a desirable economic or game-theoretic interpretation. For the stopping game of the present paper we 
\deleted[id=kri,remark={We now express this a bit differently, due to a referee comment, see also below.}]{have chosen the following equilibrium definition \eqref{eqdef2} in order to allow for the following interpretation}\added[id=kri,remark={}]{would like to define an equilibrium allowing for the following interpretation}: 
\emph{If each $x$-agent makes her stopping decision in accordance with the equilibrium stopping time $\hat \tau$, then no $x$-agent would prefer to deviate from $\hat \tau$, in the sense that,}

\begin{enumerate}[label=(\roman*)]
\item \label{item:mot0}  \emph{no $x$-agent would prefer to use a different intensity at the present $x$ than the one prescribed by $\hat \tau$; in particular,} 
\item \label{item:mot1} \emph{no $x$-agent would prefer to stop when $\hat \tau$ prescribes continuing, and}
\item \label{item:mot2} \emph{no $x$-agent would prefer to continue when $\hat \tau$ prescribes stopping.}
\end{enumerate}
\added[id=kri,remark={}]{Let us first consider}
\deleted[id=kri,remark={}]{Note that for} 
the alternative equilibrium condition,
\begin{align}
J_{\hat\tau}(x)\geq f(x) + g(h(x))\deleted[id=kri,remark={}]{,}\added[id=kri,remark={}]{.} \label{alt-eq-def}
\end{align}
\added[id=kri,remark={}]{Although \eqref{alt-eq-def} allows for the interpretations 
\ref{item:mot0}-\ref{item:mot2} it does so partly because an $x$-agent who deviates at $x$ by not stopping immediately when ${\hat\tau}$ prescribes stopping immediately does not affect the actual outcome, because $X$ has continous paths.
Hence, $X$ having continous paths implies that the strategy of stopping immediately at each $x$ is always an equilibrium under condition \eqref{alt-eq-def}.  
An advantage with the equilibrium condition of the present paper \eqref{eqdef2} compared to that of \eqref{alt-eq-def} is that it allows the $x$-agent the possibility of deviating by continuing in a way which generally \emph{does} affect the actual outcome, even though $X$ has continous paths.} 
\deleted[id=kri,remark={}]{we find that item \ref{item:mot1} of the interpretation is warranted whereas there is nothing which warrants item \ref{item:mot2}. In particular, in order to give an $x$-agent the possibility of deviating by \emph{continuing} we need to give her the possibility of deciding what happens to the process $X$ on some interval around $x$. In other words, it does not make sense to say that 'the $x$-agent continues at $x$' unless $X$ avoids stopping on some \emph{interval around} $x$.}\added[id=kri,remark={}]{Note, however, that for \eqref{eqdef2} -- as well as for other first-order equilibrium conditions -- holds that we cannot know whether an equilibrium corresponds to a maximum or another type of stationary point. This was in the context of time-inconsistent stochastic control noted in \cite[Remark 3.5]{tomas-continFORTH} and the reason in our setting is of course that the numerator in \eqref{eqdef2} can be negative for each fixed $h>0$ and still be in line with \eqref{eqdef2} by vanishing with order 
$o(\E_x(\tau_h))$. Hence, in general \ref{item:mot0}--\ref{item:mot2} are reasonable interpretations of \eqref{eqdef2} only in a restricted sense consistent with this observation and corresponding intuitively to e.g. an $x$-agent's criterion for preferring 
to stop when an equilibrium $\hat\tau$ prescribes stopping is that the instantaneous expected rate of change (relative to $\E_x(\tau_h)$) obtained by deviating is necessarily non-positive.} 
\added[id=kri,remark={}]{Using the observation in \cite[Remark 3.5]{tomas-continFORTH} as a starting point \cite{huang2018strong} introduces -- in a time-inconsistent stochastic control framework -- the notion of \emph{strong equilibrium}, which adapted to the problem of the present paper corresponds to the condition that there should exist a fixed $\bar{h}>0$  such that for each $h \in [0,\bar{h}]$ holds that the numerator of \eqref{eqdef2} is non-negative. The notion of strong equilibrium for time-inconsistent control is also studied in \cite{he2018dynamic,he2019Equilibrium}.}

\deleted[id=kri,remark={}]{With the above in mind we find the interpretation if equilibrium condition \eqref{eqdef2} ensures no $x$-agent wants to deviate from the equilibrium strategy $\hat \tau$ by using a stopping time $\tau^{\eta,D}$ during the infinitesimally short time interval $[0,\tau_h]$ as long as every other agent plays  $\hat \tau$; neither by stopping when $\hat \tau$ prescribes continuing, nor by continuing when $\hat \tau$ prescribes stopping, nor by using a different intensity than the one prescribed by $\hat \tau$. \deleted[id=kri,remark={}]{Hence, the interpretation of Definition \ref{def:equ_stop_time} is in line with \ref{item:mot0}--\ref{item:mot2}.}}

Mixed equilibria for time-inconsistent stopping are also considered in \cite{bayraktar2018} in which a mean-variance problem and a mean-standard deviation problem are studied in a discrete time Markovian setting. Pure Markov stopping times are, in analogy with the present paper, defined as entry times into sets in the state space. The definition of mixed strategy stopping times (there also called \added[id=kri,remark={}]{time-homogeneous randomized stopping times}\deleted[id=kri,remark={}]{equilibrium randomized stopping strategies}) 
\deleted[id=kri,remark={}]{and }\added[id=kri,remark={}]{is in line with the definition of Example \ref{motiv-exa} and} the definition of equilibrium 
\deleted[id=kri,remark={}]{are }
\added[id=kri,remark={}]{is a} discrete time version\deleted[id=kri,remark={}]{s} of the\deleted[id=kri,remark={}]{corresponding} definition\deleted[id=kri,remark={}]{s} in the present paper. As in the present paper, the authors find that mixed equilibria coincide with pure equilibria for the mean-variance problem. 
In \cite{huang2017optimalDISC} non-exponential discounting is studied in a discrete time Markovian setting. The considered stopping times are analogous to the pure stopping strategies $\tau^C$ of the present paper. The definition of equilibrium is a discrete time version of the equilibrium definition in the present paper.
In \cite{huang2018time}, non-exponential discounting is studied in an It\^{o} diffusion setting. Also here the considered stopping times are analogous to the pure stopping strategies of the present paper. 
\added[id=kri,remark={}]{An equilibrium in \cite{huang2018time} is defined as a fixed point of an operator $\Theta$ which describes the game-theoretic reasoning of each agent, where intuitively $\Theta$ takes as input a proposed stopping policy and produces as output the best responses of each agent; see \cite[Sec. 3.1]{huang2018time} and in particular \cite[Definition 3.7]{huang2018time}. In particular, it holds that deviation from a proposed equilibrium strategy at a particular starting point $x$ does not change the outcome when the underlying process has continuous paths and that stopping immediately at each $x$ is always an equilibrium in this case, cf. \cite[Remark 3.9]{huang2018time}; however, we remark that if the underlying process has jumps then the strategy of stopping immediately at each $x$ may no longer be an equilibrium.}
\deleted[id=kri,remark={}]{The equilibrium definition differs from the one in the present paper in the sense that it does not offer the interpretation that an $x$-agent always has the possibility to deviate by continuing, in particular it holds that \emph{stopping immediately at each state} is always an equilibrium [20,Remark 3.9].}In \cite{huang2017optimal}, non-exponential discounting and the strategies and equilibrium  of \cite{huang2018time} are studied for a more general one-dimensional Markovian process. An optimality criterion for equilibria is also proposed and studied. 
In \cite{huang2017stopping} a general time-inconsistent stopping problem is studied in the setting of a 
\added[id=kri,remark={This paper has been updated}]{strong Markov process}
\deleted[id=kri,remark={}]{stationary state process}. Both \emph{naive agents}, who continuously re-optimize, and \emph{sophisticated agents}, i.e. the approach of the present paper, are studied. The strategies and the equilibrium for the sophisticated agents are defined as in \cite{huang2018time} and immediate stopping is always an equilibrium, cf. \cite[Remark 2.6]{huang2017stopping}.

\added[id=kri,remark={}]{The equilibrium definitions of e.g. \cite{bayraktar2018,huang2017optimalDISC} are, as we have mentioned, discrete time versions of the equilibrium in the present paper and \cite{christensen2017finding}, but they are also discrete time versions of the equilibrium in e.g. \cite{huang2018time}, and the discrete time equilibrium definition seems unanimous.}

\begin{rem} The choice of equilibrium definition is a modeling choice which should be made in accordance with the economic or game-theoretic interpretation that one wants. The existence of essentially different equilibrium definitions in the literature is therefore natural.
\end{rem}

\begin{rem} \label{usual-eq} The equilibrium condition \eqref{eqdef2} is in line with the one in \cite{christensen2017finding} and inspired by time-inconsistent stopping problems in financial economics, see e.g. \cite{ebert2017discounting}. It can also be seen as an adaptation of the usual equilibrium definition for time-inconsistent stochastic control problems, see \cite{tomas-continFORTH,tomas-disc,lindensjo2017timeinconHJB} and the references therein; the main similarity between these two equilibrium definitions is that they can be said to be first order conditions, with the interpretation that the $x$-agent decides what happens on an infinitesimal interval around her. In the case of stochastic control of a diffusion this is necessary since changing the control only at a point has no effect on the control process. Note however, that the reason we have chosen a first-order equilibrium condition in the present paper is not because of mathematical\deleted[id=kri,remark={}]{ly} necessity, but because it corresponds to the interpretation that we want, 
as \deleted[id=kri,remark={}]{discussion}\added[id=kri,remark={}]{discussed} above. \added[id=kri,remark={}]{We also mention that since the selection of mixed strategy stopping times $\tau^{\lambda,C}$ involves control of $\lambda$ follows that the stopping problem of the present paper becomes in this sense also a control problem.}
\end{rem}

\begin{rem} \label{touzirem} To avoid confusion we want to reiterate that a mixed strategy stopping time $\tau^{\lambda,C}$ is in the present paper \emph{not} a strategy of an agent in the game, instead it is a specification of the strategies of \emph{all} agents in the game (i.e. a strategy profile). Similarly, an equilibrium strategy $\hat\tau$ is a full specification of the strategies of all agents.
\end{rem}

\begin{rem} \label{newremxyxy}
\added[id=kri,remark={This paragraph was previously included in Section 1.1 and has been moved here in order to coherently consider a comment of a referee}]{In \cite{touzi2002continuous} a continuous-time Dynkin game with mixed strategies defined as randomized stopping times is studied. It is instructive to note that the number of players in the game of this paper is finite while the number of players in the game of the present paper is uncountable; in the framework of the present paper it is the Cox process construction of mixed stopping strategies that makes it possible to identify mixed equilibrium strategies.} 
\added[id=kri,remark={}]{In order to choose an appropriate mathematical definition of mixed strategy one must consider the particular game being studied. In particular, the definition of a mixed strategy should have the interpretation that the agents in the game that is being studied use randomization to select pure strategies. It is therefore not surprising that the definition for a mixed strategy in e.g. \cite{touzi2002continuous} is different from that of the present paper. 
Intuitively, in the game of \cite{touzi2002continuous} each of the two agents uses randomization to select a stopping time and the mixed strategy profile that they jointly select is hence a pair of randomized stopping times; whereas in the present paper all $x$-agents jointly select a mixed strategy profile in the form of a stopping time of the kind $\tau^{\lambda,C}$.}
\end{rem}

\begin{rem} 
The type of stopping time we consider in Definition \ref{def:mix_strat} can be seen \added[id=kri,remark={}]{to be} a particular type of randomized stopping time, \added[id=kri,remark={}]{in particular $\tau^{\lambda}$ can be identified with the randomized stopping time $\inf\{t\geq0: \int_0^t\lambda(X_s)ds\geq E_1\}$ where $E_1$ is a unit exponential random variable independent of $X$} see e.g. \cite[Sec. 2]{lando1998cox}. 
\end{rem}

\section{Equilibrium conditions}\label{sec:equilibrium-conditions} 
This section contains a characterization of the equilibrium, see
Theorem \ref{main-thm-imm-stp}. It also contains other necessary and sufficient conditions for equilibrium, see
Theorem \ref{smoothfit-thm}, 
Theorem \ref{sufficient-cond-thm} and
Theorem \ref{nec-cond-THM-for-C}. 
These are the main results of the present paper. They rely on the results found in the appendix which mainly contain explicit expressions for the type of limit that is found in the left side of the equilibrium condition \eqref{eqdef2} for different values of the initial state, see  Lemma \ref{prop-calc-eq-cond1}, Lemma \ref{prop-calc-eq-cond2} and Lemma \ref{smooth-fit-lemma}. The results in the appendix rely to a large extent on arguments similar to those in the proof of Lemma \ref{main-lemma} and standard Taylor expansion.  
Theorem \ref{smoothfit-thm} and Theorem \ref{sufficient-cond-thm} rely on Proposition \ref{newprop}.

\begin{lem} \label{main-lemma} For any $\tau^{\lambda,C},\tau^{\eta,D} \in \mathcal{N}$ and $x \in D$,
\begin{align*}
\lim_{h\searrow 0}\frac{\phi_{{\tau^{\lambda,C}}\circ \theta_{\tau_h}+\tau_h}(x)-\phi_{\tau^{\lambda,C} \diamond \tau^{\eta,D}(h)}(x)}{\E_x(\tau_h)} = \eta(x)(\phi_{\lambda,C}(x)-f(x)).
\end{align*}  
\end{lem}

\begin{proof} Recall that $D$ is open by definition of $\mathcal{N}$. This implies that for any $x\in D$ there exists a constant $\bar{h}>0$ such that $\tau_h< \tau^D$ for each $0<h\leq\bar{h}$ (a.s.). Hence, for $0<h\leq\bar{h}$, 
\begin{align}
\tau^{\lambda,C} \diamond \tau^{\eta,D}(h) 
& = I_{\{\tau^{\eta,D} \leq \tau_h\}}\tau^{\eta,D} + I_{\{\tau^{\eta,D} > \tau_h\}}({\tau^{\lambda,C}}\circ \theta_{\tau_h}+\tau_h)\\
& = I_{\{\tau^{\eta} \leq \tau_h\}}\tau^{\eta} + I_{\{\tau^{\eta} > \tau_h\}}({\tau^{\lambda,C}}\circ \theta_{\tau_h}+\tau_h).
\end{align}
It follows that 
\begin{align} 
f\left(X_{\tau^{\lambda,C} \diamond \tau^{\eta,D}(h)}\right) 
& =I_{\{\tau^{\eta,D} \leq \tau_h\}}f\left(X_{\tau^{\lambda,C} \diamond \tau^{\eta,D}(h)}\right)+
I_{\{\tau^{\eta,D} > \tau_h\}}f\left(X_{\tau^{\lambda,C} \diamond \tau^{\eta,D}(h)}\right)\\
&= I_{\{\tau^\eta \leq \tau_h\}}f\left(X_{\tau^\eta}\right) +  I_{\{\tau^\eta > \tau_h\}}f\left(X_{{\tau^{\lambda,C}}\circ \theta_{\tau_h}+\tau_h}\right).
\label{dev-strat-prop}
\end{align}
Using the above, the properties of the Poisson process and by conditioning on the filtration generated by $X$, we obtain, for $0<h\leq \bar{h}$, (here $\eta_t:= \eta(X_t)$) 
\begin{align*}
& \phi_{{\tau^{\lambda,C}}\circ \theta_{\tau_h}+\tau_h}(x)-\phi_{\tau^{\lambda,C} \diamond \tau^{\eta,D}(h)}(x)\\
\enskip &= \E_x\left(
I_{\{\tau^\eta \leq \tau_h\}}
\left(f\left(X_{{\tau^{\lambda,C}}\circ \theta_{\tau_h}+\tau_h}\right)-f\left(X_{\tau^\eta}\right)
\right)\right)\\
\enskip &= \E_x\left(\int_0^\infty\eta_te^{-\int_0^t\eta_sds}
I_{\{t\leq \tau_h\}}
\left(f\left(X_{{\tau^{\lambda,C}}\circ \theta_{\tau_h}+\tau_h}\right)-f\left(X_{t}\right)
\right)dt\right)\\
\enskip &= 
\E_x\left(
f\left(X_{{\tau^{\lambda,C}}\circ \theta_{\tau_h}+\tau_h}\right)
\int_0^{\tau_h}\eta_t e^{-\int_0^t\eta_s ds}dt
-\int_0^{\tau_h}\eta_t e^{-\int_0^t\eta_s ds}f\left(X_t\right)dt
\right).
\end{align*}
By conditioning on $\mathcal{F}_{\tau_h}$ and the strong Markov property we thus obtain 
\begin{align}
& \phi_{{\tau^{\lambda,C}}\circ \theta_{\tau_h}+\tau_h}(x)-\phi_{\tau^{\lambda,C} \diamond \tau^{\eta,D}(h)}(x)\\
  &=\E_x\left(\int_0^{\tau_h}\eta_t e^{-\int_0^t\eta_s ds}dt\E_x\left(
f\left(X_{{\tau^{\lambda,C}}\circ \theta_{\tau_h}+\tau_h}\right)
|\mathcal{F}_{\tau_h}\right)-\int_0^{\tau_h}\eta_t e^{-\int_0^t\eta_s ds}f\left(X_t\right)dt\right)\\
  &=\E_x\left(\int_0^{\tau_h}\eta_t e^{-\int_0^t\eta_s ds}dt \phi_{\lambda,C}\left(X_{\tau_h}\right)-\int_0^{\tau_h}\eta_t e^{-\int_0^t\eta_s ds}f\left(X_t\right)dt\right)\\
  &=\E_x\left(\int_0^{\tau_h}\eta_t e^{-\int_0^t\eta_s ds}(\phi_{\lambda,C}\left(X_{\tau_h}\right)-f\left(X_t\right))dt\right).
\label{proof-ref}
\end{align} 
Now use the continuity of the functions $f,\eta$, $\phi_{\lambda,C}$ and the paths of $X$, and that $X$ is bounded on $[0,\tau_h]$, to obtain 
\begin{align}
& \lim_{h\searrow 0}\frac{\phi_{{\tau^{\lambda,C}}\circ \theta_{\tau_h}+\tau_h}(x)-\phi_{\tau^{\lambda,C}\diamond \tau^{\eta,D}(h)}(x)}{\E_x(\tau_h)} \\
&\enskip = \lim_{h\searrow 0}\frac{\E_x\left(\int_0^{\tau_h}\eta(X_t)e^{-\int_0^t\eta(X_s)ds}\left(\phi_{\lambda,C}\left(X_{\tau_h}\right)-f\left(X_{t}\right)
\right)dt\right)}{\E_x(\tau_h)}\\
&\enskip= \eta(x)\left(\phi_{\lambda,C}(x)-f(x)\right). \label{limit-oct}
\end{align}
The last step in \eqref{limit-oct} contains a type of limit taking that is used throughout the paper, but which we prove only in what follows. In the rest of the proof we suppress the sub-index $\lambda,C$. First note that, 
\begin{align*}
&\left|\frac{\E_x\left(\int_0^{\tau_h}\eta(X_t)e^{-\int_0^t\eta(X_s)ds}\left(\phi\left(X_{\tau_h}\right)-f\left(X_{t}\right)
\right)dt\right)}{\E_x(\tau_h)} - \eta(x)\left(\phi(x)-f(x)\right)\right|\\
&\leq\frac{\left|\E_x\left(\int_0^{\tau_h}
\eta(X_t)\left(1-e^{-\int_0^t\eta(X_s)ds}\right)
\left(\phi\left(X_{\tau_h}\right)-f\left(X_{t}\right)\right) 
dt\right)\right|}{\E_x(\tau_h)}\\
&+\frac{\left|\E_x\left(\int_0^{\tau_h}\left( 
\eta(X_t)\left(\phi\left(X_{\tau_h}\right)-f\left(X_{t}\right)\right)
-\eta(x)\left(\phi(x)-f(x)\right) 
\right)dt\right)\right|}{\E_x(\tau_h)}.
\end{align*}
From the continuity of the functions $\eta,\phi,f$ and trajectories of $X$ follows, 
\begin{align*}
&\frac{\left|\E_x\left(\int_0^{\tau_h} \left(
\eta(X_t)\left(\phi\left(X_{\tau_h}\right)-f\left(X_{t}\right)\right)
-\eta(x)\left(\phi(x)-f(x)\right) \right)
dt\right)\right|}{\E_x(\tau_h)}\\
&\leq \sup_{y_1,y_2\in (x-h,x+h)}
\left|   
\eta(y_1)\left(\phi\left(y_2\right)-f\left(y_1\right)
\right) 
-\eta(x)\left(\phi(x)-f(x)\right)
\right| \rightarrow 0, \enskip \mbox{ as } \enskip h \searrow 0.
\end{align*}
Moreover, for 
$c_h:=\sup_{y_1,y_2\in (x-h,x+h)}\left|\eta(y_1)\left(\phi\left(y_2\right)-f\left(y_1\right)\right)\right|$ holds, 
\begin{align*}
&\frac{\left|\E_x\left(\int_0^{\tau_h}
\eta(X_t)\left(1-e^{-\int_0^t\eta(X_s)ds}\right)
\left(\phi\left(X_{\tau_h}\right)-f\left(X_{t}\right)\right) dt\right)\right|}{\E_x(\tau_h)}\\
&\leq c_h\frac{\left|\E_x\left(\int_0^{\tau_h}\left(1-e^{-\int_0^{\tau_h}\eta(X_s)ds}\right)dt\right)\right|}{\E_x(\tau_h)}\\
&= c_h\frac{\left|\E_x\left(   \left(1-e^{-\int_0^{\tau_h}\eta(X_s)ds}\right) \tau_h     \right)\right|}{\E_x(\tau_h)}\\
&\leq c_h\frac{\left|\E_x\left(  \tau_h   \int_0^{\tau_h}\eta(X_t)dt   \right)\right|}{\E_x(\tau_h)}\\
&\leq c_h \sup_{y\in (x-h,x+h)} \eta(y) \frac{\left|\E_x\left(   \tau_h^2  \right)\right|}{\E_x(\tau_h)} 
\rightarrow 0, \enskip \mbox{ as } \enskip h \searrow 0,
\end{align*}
where we used the fact that $1-e^{-y}\leq y$ for $y\geq 0$ and Lemma \ref{lemma-tau-h-oct}. From the observations above follows the last step in \eqref{limit-oct}. 
\end{proof}

We are now ready to present the first main result, which characterizes the equilibrium. 
\begin{thm} \label{main-thm-imm-stp}  A stopping time $\tau^{\lambda,C} \in \mathcal{N}$ is an equilibrium stopping time if and only if it is a solution to the following system, 
\begin{align}
J_{{\lambda,C}}(x) - f(x) - g(h(x))&\geq 0, \enskip \mbox{ for } x \in C, \label{mainthmcond1} \tag{I}\\
A_Xf(x) + g'(h(x))A_X h(x)  &\leq 0, \enskip \mbox{ for } x\in \mbox{int} (C^c),\label{mainthmcond2} \tag{II}\\
f(x)-\phi_{\lambda,C}(x)+ g'(\psi_{\lambda,C}(x))\left(h(x)-\psi_{\lambda,C}(x)\right)  &=0, 
\enskip \mbox{ for $x  \in C$  with $\lambda(x)>0$}, \label{mainthmcond3} \tag{III}\\
 f(x)-\phi_{\lambda,C}(x)+ g'(\psi_{\lambda,C}(x))\left(h(x)-\psi_{\lambda,C}(x)\right) &\leq 0, 
\enskip \mbox{ for $x  \in C$  with $\lambda(x)=0$}, \label{mainthmcond4}\tag{IV}\\
\liminf_{h\searrow 0}\frac{-a(x,h)}{\E_x(\tau_h)}
&\geq 0,
\enskip \mbox{for $x \in \partial C $,}\label{mainthmcond5}\tag{V}
\end{align}
where
\begin{align}
a(x,h):= 
\E_x\left(\phi_{ {\lambda,C}}(X_{\tau_h})\right)-\phi_{ {\lambda,C}}(x)
+ g\left( \E_x\left(\psi_{ {\lambda,C}}(X_{\tau_h})\right)\right)-g(\psi_{ {\lambda,C}}(x)). \label{newlabelKR}
\end{align}
\end{thm}
See the appendix for a proof of Theorem \ref{main-thm-imm-stp}. We will use the following general result.

\begin{prop} \label{newprop} Consider a fixed $x\in \state$ and a function $k:E\rightarrow \mathbb{R}$. Suppose that there exists a constant $\bar{h}>0$ such that $k$ is $\mathcal{C}^2$ on $[x-\bar{h},x]$ and $[x,x+\bar{h}]$ and continuous on $[x-\bar{h},x+\bar{h}]$, then
\begin{align}
\lim_{h\searrow 0}
\frac{\left( \E_x(k(X_{\tau_h}))-k(x)
\right)^2}{\E_x(\tau_h)}
= \left(\frac{k'(x+) - k'(x-)}{2}\right)^2 \sigma^2(x).
\end{align}
In particular, for the local time of $X$ at $x$, denoted by $l^x_t(X)$, it holds that 
\begin{align}
\lim_{h\searrow 0}\frac{\E_x\left(l^x_{\tau_h}(X)\right)^2}{\E_x(\tau_h)}= \sigma^2(x).
\end{align}
\end{prop}
\begin{proof}
Use the It\^{o}-Tanaka formula, see e.g. \cite{peskir2005change} or \cite[p. 75]{peskir2006optimal}, to obtain, for $0<h\leq\bar{h}$,
\begin{align}
k(X_{\tau_h}) - k(x) 
 &= 
\int_0^{\tau_h}A_Xk(X_t)I_{\{X_t \neq x\}}dt   +\int_0^{\tau_h}k'(X_t)\sigma(X_t)I_{\{X_t \neq x\}}dW_t\\
& \quad + \frac{1}{2}\int_0^{\tau_h}\left(k'(X_t+) - k'(X_t-)\right)I_{\{X_t = x\}}dl^x_t(X)\\
& = 
\int_0^{\tau_h}A_Xk(X_t)I_{\{X_t \neq x\}}dt +\int_0^{\tau_h}k'(X_t)\sigma(X_t)I_{\{X_t \neq x\}}dW_t\\
& \quad  + \frac{1}{2}\left(k'(x+) - k'(x-)\right) l^x_{\tau_h}(X).
\end{align}
Thus,
\begin{align}
& \lim_{h\searrow 0}\frac{\left( \E_x(k(X_{\tau_h}))-k(x)\right)^2}{\E_x(\tau_h)}\\
&\enskip=\lim_{h\searrow 0}\frac{\left(\E_x\left(
\int_0^{\tau_h}A_Xk(X_t)I_{\{X_t \neq x\}}dt\right) + 
\frac{1}{2}\left(k'(x+) - k'(x-)\right) \E_x\left(l^x_{\tau_h}(X)
\right)\right)^2}{\E_x(\tau_h)}. \label{newprop-okt1}
\end{align}
Observe that
$\lim_{h\searrow 0}\E_x\left(l^x_{\tau_h}(X)\right) = 0$,  
$\lim_{h\searrow 0}\E_x\left(\int_0^{\tau_h}A_Xk(X_t)I_{\{X_t \neq x\}}dt\right)=0$ and
$\lim_{h\searrow 0}\frac{\E_x\left(\int_0^{\tau_h}A_Xk(X_t)I_{\{X_t \neq x\}}dt\right)}{\E_x(\tau_h)}$ is finite. 
Thus, expansion of the square in \eqref{newprop-okt1} gives, 
\begin{align}
\lim_{h\searrow 0}\frac{\left( \E_x(k(X_{\tau_h}))-k(x)\right)^2}{\E_x(\tau_h)}= \left(\frac{\left(k'(x+) - k'(x-)\right)}{2}\right)^2\lim_{h\searrow 0}\frac{\E_x\left(l^x_{\tau_h}(X)\right)^2}{\E_x(\tau_h)}.\label{newprop1}
\end{align}
Applying the result in \eqref{newprop1} for $k(y):=|y-x|$ (recall that $x$ is fixed) gives us
\begin{align}
\lim_{h\searrow 0}\frac{\left( \E_x(|X_{\tau_h}-x|)\right)^2}{\E_x(\tau_h)}
& = \left(\frac{\left(1 - (-1)\right)}{2}\right)^2\lim_{h\searrow 0}\frac{\E_x\left(l^x_{\tau_h}(X)\right)^2}{\E_x(\tau_h)}\\
& = \lim_{h\searrow 0}\frac{\E_x\left(l^x_{\tau_h}(X)\right)^2}{\E_x(\tau_h)}. \label{newprop2}
\end{align}
However, it is also easy to see that,
\begin{align}
\lim_{h\searrow 0}\frac{\left( \E_x(|X_{\tau_h}-x|)\right)^2}{\E_x(\tau_h)}
& = \lim_{h\searrow 0}\frac{\left(p_h|x+h-x| + (1-p_h)|x-h-x|\right)^2}{\E_x(\tau_h)}\\
& = \lim_{h\searrow 0}\frac{h^2}{\E_x(\tau_h)},\label{newprop3}
\end{align}
where $p_h:= \mathbb{P}_x(X_{\tau_h}=x+h)$. The result follows from
\eqref{newprop1},
\eqref{newprop2}, 
\eqref{newprop3} and the following limit which is proved in Lemma \ref{lemma-tau-h-oct}, 
\begin{align}
\lim_{h\searrow 0} \frac{h^2}{\E_x(\tau_h)} = \sigma^2(x)\label{newprop4}.
\end{align}
\end{proof}

\begin{rem}  The limit \eqref{newprop4} was also recently proved in \cite{geiss2017first}. 
 In Lemma \ref{lemma-tau-h-oct} we prove \eqref{newprop4} using It\^{o}'s formula and the optional sampling theorem. In \cite{geiss2017first} the proof of \eqref{newprop4} relies on a representation of the denominator based on the scale function and the speed measure of the diffusion $X$ and standard limit arguments, e.g. l'Hospital's rule. 
\end{rem}

Theorem \ref{smoothfit-thm} below presents a smooth fit condition that an equilibrium value function must satisfy at any $x \in \partial C$, under additional assumptions. We use this result when making an ansatz to finding an equilibrium stopping time in Section \ref{applications:mean-var}.

\begin{thm} \label{smoothfit-thm}  Suppose that $\tau^{\lambda,C}$ is an equilibrium stopping time. For a fixed $x\in \partial C$, if the functions 
$\phi_{ {\lambda,C}}$ and $\psi_{ {\lambda,C}}$ are $\mathcal{C}^2$ on
$[x-\bar{h},x]$ and $[x,x+\bar{h}]$ for some constant $\bar{h}>0$, then the equilibrium value function $J_{\lambda,C}$ satisfies smooth fit in the sense that
\begin{align}
J_{\lambda,C}'(x) = f'(x) + g'(h(x))h'(x). 
\end{align} 
\end{thm}
\begin{proof}
Consider a fixed $x\in \partial C$. For any $\epsilon$, satisfying $x+\epsilon \in E$ (both negative and positive such $\epsilon$ exist since $E$ is open and $\partial C$ is the boundary of $C$ in $E$), it holds that
\begin{align}
J_{\lambda,C}(x +\epsilon ) \geq f(x+\epsilon) + g(h(x+\epsilon)).
\end{align}
To see this use that this inequality is an equality when $x+\epsilon \notin C$, and condition \eqref{mainthmcond1} for the case $x+\epsilon \in C$. Moreover, since $x\in \partial C$ it follows that $J_{\lambda,C}(x) = f(x) + g(h(x))$. Hence,
\begin{align}
J_{\lambda,C}(x +\epsilon ) - J_{\lambda,C}(x) \geq f(x+\epsilon) - f(x)+ g(h(x+\epsilon))- g(h(x)).
\end{align}
If $\epsilon<0$ it follows that
\begin{align}
\frac{J_{\lambda,C}(x+\epsilon)- J_{\lambda,C}(x)}{\epsilon}
\leq \frac{f(x+\epsilon) - f(x) }{\epsilon}  + \frac{g(h(x+\epsilon))- g(h(x))}{\epsilon}. \label{smoothfit-thm:eq1}
\end{align}
Hence, the left derivative satisfies
$
J'^{(-)}_{\lambda,C}(x) \leq f'(x) + g'(h(x))h'(x).
$
The right derivative can be similarly dealt with and we thus obtain
\begin{align} 
 J'^{(-)}_{\lambda,C}(x) \leq f'(x) + g'(h(x))h'(x) \leq  J'^{(+)}_{\lambda,C}(x). \label{smoothfit-thm:eq1.5}
\end{align}
Let us now prove that if we would not have smooth fit then condition \eqref{mainthmcond5} would be violated and hence smooth fit must hold, by Theorem \ref{main-thm-imm-stp}. Note that if smooth fit would not hold then
$
J_{\lambda,C}'^{(+)}(x)-J_{\lambda,C}'^{(-)}(x)>0,
$
cf. \eqref{smoothfit-thm:eq1.5}, which is equivalent to
\begin{align}
\phi_{\lambda,C}'(x+) + g'(\psi_{\lambda,C}(x))\psi'_{\lambda,C}(x+) 
>\phi_{\lambda,C}'(x-) + g'(\psi_{\lambda,C}(x))\psi_{\lambda,C}'(x-).
\end{align}
To see this use that $J_{\lambda,C}(x)=\phi_{\lambda,C}(x)+g(\psi_{\lambda,C}(x))$ and the chain rule, and then the differentiability assumptions (i.e. $\phi_{\lambda,C}$ and $\psi_{\lambda,C}$ are $\mathcal{C}^2$ on $[x-\bar{h},x]$ and $[x,x+\bar{h}]$) and continuity (of $\phi_{\lambda,C}$ and $\psi_{\lambda,C}$, cf. admissibility, Definition \ref{def:mix_strat}). Rewrite the equation above as
\begin{align}
\phi_{\lambda,C}'(x+)-\phi_{\lambda,C}'(x-) + g'(\psi_{\lambda,C}(x))(\psi'_{\lambda,C}(x+) -\psi'_{\lambda,C}(x-))>0.
\label{contradic-assum}
\end{align}
The differentiability assumptions imply that we can use the It\^{o}-Tanaka formula to obtain, for $0<h<\bar{h}$,
\begin{align}
& \phi_{\lambda,C}(X_{\tau_h}) - \phi_{\lambda,C}(x)\\ 
& = 
\int_0^{\tau_h}A_X\phi_{\lambda,C}(X_t)I_{\{X_t \neq x\}}dt +\int_0^{\tau_h}\phi'_{\lambda,C}(X_t)\sigma(X_t)I_{\{X_t \neq x\}}dW_t\\
& \quad  + \frac{1}{2}\left(\phi'_{\lambda,C}(x+) - \phi'_{\lambda,C}(x-)\right) l^x_{\tau_h}(X).\label{ito-tanaka-formula}
\end{align}
Thus, 
\begin{align}
\E_x\left(\phi_{\lambda,C}(X_{\tau_h})\right) - \phi_{\lambda,C}(x) = \enskip a_1(h) + a_2 \E_x\left(l^x_{\tau_h}(X)\right),
\end{align}
for
$a_1(h):=\E_x\left(\int_0^{\tau_h}A_X\phi_{\lambda,C}(X_t)I_{\{X_t \neq x\}}dt\right)$ and
$a_2:=\frac{1}{2}(\phi'_{\lambda,C}(x+) - \phi'_{\lambda,C}(x-))$.
Similarly,
\begin{align}
\E_x\left(\psi_{\lambda,C}(X_{\tau_h})\right) - \psi_{\lambda,C}(x) = \enskip b_1(h) + b_2\E_x\left(l^x_{\tau_h}(X)\right), \label{addedapr18}
\end{align}
for 
$b_1(h):=\E_x\left(\int_0^{\tau_h}A_X\psi_{\lambda,C}(X_t)I_{\{X_t \neq x\}}dt\right)$
and $b_2:=\frac{1}{2}(\psi'_{\lambda,C}(x+) - \psi'_{\lambda,C}(x-))$. 
Hence, using standard Taylor expansion of the function $g$ we write $a(x,h)$ in \eqref{newlabelKR} as, 
\begin{align}
a(x,h)  
&= a_1(h)  + a_2\E_x\left(l^x_{\tau_h}(X)\right)  + g'(\psi_{\lambda,C}(x))\{b_1(h) + b_2\E_x\left(l^x_{\tau_h}(X)\right)\}\\
&\enskip + \frac{1}{2}g''( \psi_{\lambda,C}(x))\{b_1(h) + b_2\E_x\left(l^x_{\tau_h}(X)\right)\}^2\\
&\enskip + \frac{1}{6}g'''(c_h)\{b_1(h) + b_2\E_x\left(l^x_{\tau_h}(X)\right)\}^3\added[id=kri,remark={}]{,}\deleted[id=kri,remark={}]{.}\label{newsmoothmar0}
\end{align}
where $c_h$ is a constant between 
$\psi_{\lambda,C}(x)$ and  
$\psi_{\lambda,C}(x)+b_1(h) + b_2\E_x\left(l^x_{\tau_h}(X)\right)$.  This can be written as,
\begin{align}
&\frac{-a(x,h)}{\E_x(\tau_h)}\\
&\enskip\quad =-\frac{a_1(h)}{\E_x(\tau_h)} - g'( \psi_{\lambda,C}(x))\frac{ b_1(h)}{\E_x(\tau_h)} \label{newsmoothmar1}\\
&\quad\quad\quad- (a_2+g'(\psi_{\lambda,C}(x))b_2)\frac{\E_x\left(l^x_{\tau_h}(X)\right)}{\E_x(\tau_h)}\label{newsmoothmar2}\\ 
&\quad\quad\quad- \frac{1}{2}g''( \psi_{\lambda,C}(x))
\frac{\left(\E_x\left(\psi_{\lambda,C}(X_{\tau_h})\right) - \psi_{\lambda,C}(x)\right)^2}{\E_x(\tau_h)} 
\label{newsmoothmar3}  \\
&\quad\quad\quad - \frac{1}{6}g'''(c_h)
\frac{\left(\E_x\left(\psi_{\lambda,C}(X_{\tau_h})\right) - \psi_{\lambda,C}(x)\right)^3}{\E_x(\tau_h)}.\label{newsmoothmar4} 
\end{align}
Let us see what happens to the liminf of $\frac{-a(x,h)}{\E_x(\tau_h)}$ when sending $h\searrow 0$: The liminf of the terms in \eqref{newsmoothmar1} are finite due to the differentiability assumptions for $\phi_{\lambda,C}$ and $\psi_{\lambda,C}$. The term in \eqref{newsmoothmar2} can be written as
\begin{align}
&- (a_2+g'(\psi_{\lambda,C}(x))b_2)\frac{\E_x\left(l^x_{\tau_h}(X)\right)}{\E_x(\tau_h)}\\
&\enskip = -\frac{1}{2}\left(
  \phi'_{\lambda,C}(x+) - \phi'_{\lambda,C}(x-)  +
g'(\psi_{\lambda,C}(x))(\psi'_{\lambda,C}(x+) - \psi'_{\lambda,C}(x-))
\right) \\   
&\quad \quad\quad\times\frac{\E_x\left(l^x_{\tau_h}(X)\right)}{\E_x(\tau_h)}.
\end{align}
From Proposition \ref{newprop} we know that $\lim_{h\searrow 0}\frac{\E_x\left(l^x_{\tau_h}(X)\right)^2}{\E_x(\tau_h)} = \sigma^2(x)$, where the limits of the numerator and the denominator are both zero, and $\sigma(x)>0$ by assumption.
Hence,
\begin{align}
\lim_{h\searrow 0}\frac{\E_x\left(l^x_{\tau_h}(X)\right)}{\E_x(\tau_h)} 
&= \lim_{h\searrow 0}\frac{\E_x\left(l^x_{\tau_h}(X)\right)^2}{\E_x(\tau_h)}\frac{1}{\E_x\left(l^x_{\tau_h}(X)\right)}\\
&=\sigma^2(x)\lim_{h\searrow 0}\frac{1}{\E_x\left(l^x_{\tau_h}(X)\right)}\\
&=\infty.
\end{align}
Thus, from the contradiction assumption \eqref{contradic-assum} follows that the liminf of the term in \eqref{newsmoothmar2} is equal to $-\infty$.
Proposition \ref{newprop} gives an explicit expression for the liminf of the term in \eqref{newsmoothmar3}, which in particular implies that this limit is finite.  
The liminf of the term in \eqref{newsmoothmar4} vanishes, to see this use that the limit of the ratio in \eqref{newsmoothmar3} is finite and $\E_x\left(\psi_{\lambda,C}(X_{\tau_h})\right) - \psi_{\lambda,C}(x)\rightarrow 0$ (cf. continuity of $\psi_{\lambda,C}$). This implies that condition \eqref{mainthmcond5}  would indeed be violated if \eqref{contradic-assum} were true and smooth fit must therefore hold.
\end{proof}

Theorem \ref{main-thm-imm-stp} presents necessary and sufficient conditions for a stopping time $\tau^{\lambda,C}$ to be an equilibrium stopping time. If we for an equilibrium stopping time candidate $\tau^{\lambda,C}$ can find explicit expressions for the functions $\phi_{\lambda,C}$ and $\psi_{\lambda,C}$ then it is easy to verify if conditions \eqref{mainthmcond1}--\eqref{mainthmcond4} hold whereas condition \eqref{mainthmcond5} is not necessarily easy to verify. Theorem \ref{sufficient-cond-thm} below presents a more easily verified characterization of condition \eqref{mainthmcond5}, given additional differentiability conditions. We will use Theorem \ref{sufficient-cond-thm} to verify an ansatz to finding an equilibrium in Section \ref{applications:mean-var}.  

\begin{thm} \label{sufficient-cond-thm}  Consider a stopping time $\tau^{\lambda,C} \in \mathcal{N}$. If for any fixed $x\in \partial C$ there exists a constant $\bar{h}>0$ such that the functions $\phi_{\lambda,C}$ and $\psi_{\lambda,C}$ are $\mathcal{C}^2$ on $[x-\bar{h},x]$ and $[x,x+\bar{h}]$ and such that the function 
$\phi_{\lambda,C}(\cdot) + g'(\psi_{\lambda,C}(x)) \psi_{\lambda,C}(\cdot)$ is $\mathcal{C}^1$  on $[x-\bar{h},x+\bar{h}]$ then condition \eqref{mainthmcond5} is equivalent to,
\begin{align} 
&  A_X\phi_{\lambda,C}(x+) + g'( \psi_{\lambda,C}(x))A_X\psi_{\lambda,C}(x+) 
+  A_X\phi_{\lambda,C}(x-) + g'( \psi_{\lambda,C}(x))A_X\psi_{\lambda,C}(x-)\\ 
& \enskip\enskip + g''( \psi_{\lambda,C}(x))\left(\frac{\psi'_{\lambda,C}(x+) - \psi'_{\lambda,C}(x-)}{2}\right)^2 
\sigma^2(x)\leq 0.\label{sufficient-cond-thm:1} 
\end{align} 
\end{thm}
\begin{proof}  
Consider an arbitrary $x \in \partial C$. Use the It\^{o}-Tanaka formula to arrive at the same expression as in 
\eqref{newsmoothmar1}-- \eqref{newsmoothmar4}. Note that the $\mathcal{C}^1$ assumption in the statement of the theorem directly implies that $a_2 + g'(\psi_{\lambda,C}(x))b_2 =0$. This implies, using \eqref{newsmoothmar1}--\eqref{newsmoothmar4}, that the expression that we take the limit of in \eqref{mainthmcond5} can be written as %
\begin{align}
&\frac{-a(x,h)}{\E_x(\tau_h)}\\
& =-\frac{
\E_x\left(\int_0^{\tau_h}\left(
A_X\phi_{\lambda,C}(X_t)I_{\{X_t \neq x\}}
+ g'( \psi_{\lambda,C}(x))A_X\psi_{\lambda,C}(X_t)I_{\{X_t \neq x\}}
\right)dt
\right) 
}{\E_x(\tau_h)}  \label{newsuff1}\\
&\quad- \frac{1}{2}g''( \psi_{\lambda,C}(x))\frac{\left(
\E_x\left(\psi_{\lambda,C}(X_{\tau_h})\right) - \psi_{\lambda,C}(x) 
\right)^2}{\E_x(\tau_h)} 
-... \label{newsuff2}
\end{align}
where the last term, which has been notationally suppressed, converges to zero as $h\searrow 0$ (cf. the end the proof of Theorem \ref{smoothfit-thm}). The differentiability assumptions and basic properties of diffusions imply that 
\begin{align} 
& \lim_{h\searrow 0}\left( 
\frac{\E_x\left(\int_0^{\tau_h}\left(A_X\phi_{\lambda,C}(X_t)I_{\{X_t \neq x\}}+ g'( \psi_{\lambda,C}(x))A_X\psi_{\lambda,C}(X_t)I_{\{X_t \neq x\}}\right)dt\right)}{\E_x(\tau_h)}\right)\\
&\enskip = \frac{1}{2}\left(A_X\phi_{\lambda,C}(x+) + g'( \psi_{\lambda,C}(x))A_X\psi_{\lambda,C}(x+)\right)\\
&\quad + \frac{1}{2}\left(A_X\phi_{\lambda,C}(x-) + g'( \psi_{\lambda,C}(x))A_X\psi_{\lambda,C}(x-)\right).
\end{align}
Now use Proposition \ref{newprop} to obtain the result.
\end{proof}

Theorem \ref{nec-cond-THM-for-C} below presents a necessary condition for equilibria for $x \in C$ in the case that the equilibrium intensity function is strictly positive, under additional assumptions. This result will be used when we make an ansatz to finding an equilibrium stopping time in Section \ref{applications:var}.

\begin{thm} \label{nec-cond-THM-for-C} 
Suppose that $\tau^{\lambda,C}$ is an equilibrium stopping time with $\lambda(x)>0$ for $x\in C$ and that $\psi_{\lambda,C}$ is $\mathcal{C}^2$ on $C$. Then $\psi_{\lambda,C}$ satisfies the (non-linear) ODE 
\begin{align}
& -\left(\mu(x)\psi'_{\lambda,C}(x) + \frac{1}{2}\sigma^2(x)\psi''_{\lambda,C}(x)\right)(h(x)-\psi_{\lambda,C}(x))g''(\psi_{\lambda,C}(x))\\
& = \mu(x)\{f'(x) +  h'(x)g'(\psi_{\lambda,C}(x))\} +  \frac{1}{2}\sigma^2(x)\{f''(x) +  d(x)\}, \mbox{ for $x\in C$},\label{nec-cond-THM-for-C0.5}
\end{align}
where 
\begin{align}
d(x) := & g'''(\psi_{\lambda,C}(x))(\psi'_{\lambda,C}(x))^2\left(h(x)-\psi_{\lambda,C}(x)\right)\\
&+ 2g''(\psi_{\lambda,C}(x))\psi'_{\lambda,C}(x)(h'(x)-\psi'_{\lambda,C}(x)) + g'(\psi_{\lambda,C}(x))h''(x). \label{newcoco}
\end{align}
Moreover, the equilibrium intensity function $\lambda$ satisfies
\begin{align} 
& \lambda(x)(h(x)-\psi_{\lambda,C}(x))^2g''(\psi_{\lambda,C}(x))   \\
&= \mu(x)\{f'(x) +  h'(x)g'(\psi_{\lambda,C}(x))\} +  \frac{1}{2}\sigma^2(x)\{f''(x) +  d(x)\},  \mbox{ for $x\in C$}.
\label{nec-cond-THM-for-C0}
\end{align}
\end{thm}
\begin{proof} Suppose that $\tau^{\lambda,C}$ is an equilibrium stopping time with $\lambda(x)>0$ for $x\in C$. Consider an arbitrary fixed $x\in C$. By definition $J_{\lambda,C}(x) = \phi_{\lambda,C}(x) + g(\psi_{\lambda,C}(x))$ and hence
\begin{align} \label{nec-cond-THM-for-C-1}
A_X(J_{\lambda,C}(x)-\phi_{\lambda,C}(x) - g(\psi_{\lambda,C}(x)))=0.
\end{align}
Condition \eqref{mainthmcond3} holds by Theorem \ref{main-thm-imm-stp} and from simple calculations follows, 
\begin{align}
&J_{\lambda,C}(x)-\phi_{\lambda,C}(x) - g(\psi_{\lambda,C}(x))\\
& = f(x) -\phi_{\lambda,C}(x) + g'(\psi_{\lambda,C}(x))\left(h(x)-\psi_{\lambda,C}(x)\right).
\label{nec-cond-THM-for-C-2}
\end{align}
We will notationally suppress $\lambda,C$ and $(x)$ in the rest of the proof. From \eqref{nec-cond-THM-for-C-1} and \eqref{nec-cond-THM-for-C-2} follows that $
A_Xf - A_X\phi  + A_X\left(g'(\psi)\left(h-\psi\right)\right) = 0$
which implies that
\begin{align} 
A_X\phi  = A_Xf + A_X\left(g'(\psi)\left(h-\psi\right)\right).
\label{nec-cond-THM-for-C-3}
\end{align}
Now use Lemma \ref{main-lemma2} and then condition \eqref{mainthmcond3} to see that
\begin{align}
A_X\phi  & = \lambda(\phi-f) \\
& = \lambda g'(\psi)\left(h-\psi\right).\label{nec-cond-THM-for-C-5}
\end{align}
Let us investigate the expressions in the right side of \eqref{nec-cond-THM-for-C-3}. The assumed differentiability implies that
\begin{align}
A_Xf & = \mu f' + \frac{1}{2}\sigma^2f'', \enskip \mbox{ and}\\  
A_X\left(g'(\psi)\left(h-\psi\right)\right) &= 
\mu(g'(\psi)\left(h-\psi\right))' + \frac{1}{2}\sigma^2 (g'(\psi)\left(h-\psi\right))''.
\end{align}
Use standard differentiation rules to find that the derivatives in the last expression can be written as
\begin{align}
 (g'(\psi)\left(h-\psi\right))'  
= \psi' b +  h'g'(\psi),
\end{align}
where we use the temporary notation $b:=g''(\psi)\left(h-\psi\right)-g'(\psi)$, and 
\begin{align}
(g'(\psi)\left(h-\psi\right))''
= d + \psi'' b,
\end{align}
where $d$ is defined in \eqref{newcoco}.
It follows that the right side of \eqref{nec-cond-THM-for-C-3} can be written as
\begin{align}
& A_Xf + A_X\left(g'(\psi)\left(h-\psi\right)\right)\\ 
& = \mu\{f' +  h'g'(\psi)\} +  \frac{1}{2}\sigma^2\{f'' +  d \} + b\{\mu \psi'  + \frac{1}{2}\sigma^2 \psi''\} \\
& = \mu\{f' +  h'g'(\psi)\} +  \frac{1}{2}\sigma^2\{f'' +  d \} + bA_X\psi\\
& = \mu\{f' +  h'g'(\psi)\} +  \frac{1}{2}\sigma^2\{f'' +  d \} + b\lambda(\psi-h),
\end{align}
where we relied on Lemma \ref{main-lemma2} (which analogously holds also for the function $\psi$) and the differential operator 
form of $A_X$. 
Use the equality above, \eqref{nec-cond-THM-for-C-3} and \eqref{nec-cond-THM-for-C-5} to obtain 
\begin{align}
\lambda g'(\psi)\left(h-\psi\right) = \mu\{f' +  h'g'(\psi)\} +  \frac{1}{2}\sigma^2\{f'' +  d \} + b\lambda(\psi-h).
\end{align}
This implies that
\begin{align}
\lambda\left(h-\psi\right)\{g'(\psi) + b\} = \mu\{f' +  h'g'(\psi)\} +  \frac{1}{2}\sigma^2\{f'' +  d\}.
\end{align}
Use that $b + g'(\psi) =g''(\psi )\left(h-\psi\right)$ to see that \eqref{nec-cond-THM-for-C0} follows. Now use Lemma \ref{main-lemma2} to obtain $A_X\psi = \lambda (\psi -h)$. Using the assumed differentiability for $\psi$ we also obtain $A_X\psi = \mu\psi' + \frac{1}{2}\sigma^2\psi''.$ Hence, $\lambda (h-\psi)= -\left(\mu \psi' + \frac{1}{2}\sigma^2 \psi''\right)$ which, together with \eqref{nec-cond-THM-for-C0}, implies that \eqref{nec-cond-THM-for-C0.5} holds. 
\end{proof}

\section{Examples} \label{applications}
The main objectives of the present paper are to formulate and solve time-inconsistent stopping problems of the type \eqref{the-problem} and to define mixed strategies for these problems. 
In this section we first study a variance problem for which it turns out a mixed equilibrium but no pure equilibrium exists.
Second, we study a mean-variance problem for which it turns out a pure equilibrium or no equilibrium exists depending on the parameters, in particular there is no mixed equilibrium.
We also present a simple example with two different equilibria, showing that we cannot generally expect equilibrium uniqueness.

\subsection{A variance stopping problem} \label{applications:var}
The variance stopping problem corresponds to the time-inconsistent problem of trying to maximize 
\[\mbox{Var}_x(X_\tau).\]
An economic motivation for a variance stopping problem is found in \cite{pedersen2011explicit} and the references therein. Variance stopping problems are also studied in \cite{gad2016optimal,gad2015variance} using randomized stopping times. We also refer to \cite{buonaguidi2015remark,buonaguidi2018some}. All these references consider the problem from the perspective of the pre-commitment approach.

The variance problem is given by 
$f(x):=x^2, g(x):=-x^2$ and $h(x):=x$. To see this note that 
\begin{align}
J_{\tau}(x) &= \phi_\tau(x) + g(\psi_\tau(x))\\ 
&= \E_x(X_\tau^2) - \E_x^2(X_\tau) \\
& = \mbox{Var}_x(X_\tau). 
\end{align} 
We consider a positive state process $X$. In this case Assumption \ref{ass:general-lemmas} is satisfied. 
It follows that   
\begin{align} \label{var-help-eq}
g'(\psi_{\lambda,C}(x))= -2\psi_{\lambda,C}(x) \mbox{ and } 
J_{\lambda,C}(x) =  \phi_{\lambda,C}(x) - \psi^2_{\lambda,C}(x).
\end{align}
Hence, simple calculations yield 
\begin{align} 
& f(x) -\phi_{\lambda,C}(x) + g'(\psi_{\lambda,C}(x))\left(h(x)-\psi_{\lambda,C}(x)\right)\\
&= - (\phi_{\lambda,C}(x)-\psi^2_{\lambda,C}(x)) + x^2 -2x\psi_{\lambda,C}(x)+ \psi^2_{\lambda,C}(x)\\
&= - J_{\lambda,C}(x) + (\psi_{\lambda,C}(x)-x)^2 \label{var-1}.
\end{align}
An equilibrium stopping time should typically not recommend immediate stopping since this corresponds to minimal variance, see also Remark \ref{new-var-rem} below. Hence, we make an ansatz with $C=E$. Specifically, we make the ansatz that an equilibrium stopping time is given by $\tau^{\lambda,E}$ for some strictly positive intensity function $\lambda$ which is to be determined. We will use the notation $\tau^{\lambda,E} = \tau^{\lambda}$, $\psi_{\lambda,E}= \psi_{\lambda}$ etc.

We immediately obtain the following result.
\begin{thm} \label{thm:eq-condition} A stopping time $\tau^\lambda \in \mathcal{N}$, with $\lambda(x)>0$ for each $x\in \state$, is an equilibrium stopping time for the variance problem if and only if
\begin{align}
J_\lambda(x) &= (\psi_{\lambda}(x)-x)^2, \enskip \mbox{ for $x  \in E$} \label{equi-char}.
\end{align}
Moreover, if \eqref{equi-char}  holds then $J_\lambda$ given by \eqref{equi-char} is the corresponding equilibrium value function. 
\end{thm} 
\begin{proof} Use that $h(x)=x$, $- f(x) - g(x) = 0$ and \eqref{var-1} to see that if \eqref{equi-char} holds then \eqref{mainthmcond1} and \eqref{mainthmcond3} hold, whereas \eqref{mainthmcond2}, \eqref{mainthmcond4} and \eqref{mainthmcond5} can be considered trivially fulfilled, since we use $C=E$ and $\lambda(x)>0$. Now, if \eqref{mainthmcond3} holds then it follows from \eqref{var-1} and $C=E$ that \eqref{equi-char} holds. Thus, the first assertion follows from Theorem \ref{main-thm-imm-stp}. The second  assertion follows immediately.
\end{proof}
Let us use the ODE condition \eqref{nec-cond-THM-for-C0.5} in Theorem \ref{nec-cond-THM-for-C} to identify a candidate for $\psi_\lambda$ and then use the result \eqref{nec-cond-THM-for-C0} to identify the corresponding candidate equilibrium  intensity function $\lambda$. In the present case the ODE  \eqref{nec-cond-THM-for-C0.5} is
\begin{align}
& -\left(\mu(x)\psi'_{\lambda}(x) + \frac{1}{2}\sigma^2(x)\psi''_{\lambda}(x)\right)(x-\psi_{\lambda}(x))(-2)\\
& = \mu(x)\{2x -  2\psi_{\lambda}(x)\} +  \frac{1}{2}\sigma^2(x)\{2 +  d(x)\},
\end{align}
with 
$d(x) = 4(\psi'_{\lambda}(x))^2-4\psi'_{\lambda}(x)$, where we used \eqref{var-help-eq}, $f'(x) = 2x, g'''(x)=0$ etc. We note that if $x-\psi_{\lambda}(x) \neq 0$, then the ODE simplifies to 
\begin{align}
\mu(x)\psi'_{\lambda}(x) + \frac{1}{2}\sigma^2(x)\psi''_{\lambda}(x) 
= \mu(x) +  \frac{1}{2}\sigma^2(x)\frac{(\psi'_{\lambda}(x)-1)^2 + (\psi'_{\lambda}(x))^2}{x-\psi_{\lambda}(x)}.
\label{var-ODE}
\end{align}
In case $X$ is a geometric Brownian motion it turns out that the problem can be solved explicitly. Thus, from now we assume (in this example)  that
\begin{align}
dX_t = \mu X_t dt + \sigma X_t dW_t \label{GBM}.
\end{align}
In this case \eqref{var-ODE} becomes
\begin{align}
& \mu x (\psi'_{\lambda}(x)-1) +  
\frac{1}{2}\sigma^2x^2\left(\psi''_{\lambda}(x)-\frac{(\psi'_{\lambda}(x)-1)^2 + (\psi'_{\lambda}(x))^2}{x-\psi_{\lambda}(x)}\right) =  0.\label{var-ODE2} 
\end{align}
The ODE \eqref{var-ODE2}  has, under appropriate assumptions for the constants $\mu$ and $\sigma$, one solution (at least) on the form $\psi_{\lambda}(x)=cx$ for some constant $c \neq0,1$. To see this use that $\psi_{\lambda}''(x)=(cx)''=0$ and that $x>0$, since $E=(0,\infty)$ for the GBM. Now use \eqref{nec-cond-THM-for-C0} and the candidate $\psi_{\lambda}(x)=cx$ to obtain the corresponding candidate intensity 
\begin{align} 
\lambda(x) & = \frac{\mu(x)\{f'(x) +  h'(x)g'(\psi_{\lambda}(x))\} +  \frac{1}{2}\sigma^2(x)\{f''(x) +  d(x).\}}{(h(x)-\psi_{\lambda}(x))^2g''(\psi_{\lambda}(x))}\\
& = \frac{\mu  \{1 -c\} +  \frac{1}{2}\sigma^2 \{1 +  2c^2-2c.\}}{- (1-c)^2}.\label{var-lambda-c}
\end{align}
This means the candidate solution $\psi_{\lambda}(x)=cx$ corresponds to using a constant intensity   (depending on the constant $c$). This constant candidate intensity could, with some effort, be found by identifying the constant(s) $c$ such that $\psi_{\lambda}(x)=cx$ solves \eqref{var-ODE2}, and inserting this $c$ into \eqref{var-lambda-c} and thereby obtaining a corresponding constant equilibrium intensity candidate. We shall, however, instead use Theorem \ref{thm:eq-condition} to identify the constant equilibrium intensity (it turns out that only one constant equilibrium intensity exists) and thereby verify that the ansatz works. This is done in the proof of Theorem \ref{thm:var}.

\begin{thm}  \label{thm:var} Let $X$ be given by \eqref{GBM} where the constants $\mu$ and $\sigma$ satisfy $\sigma^2>0$ and
\begin{align} \label{parameter-cond}
2\mu+\sigma^2<0.
\end{align}
Then $\tau^{\lambda}$, with
\begin{align} \label{eq-lambda}
\lambda= \sqrt{\frac{-\mu^2(2\mu+\sigma^2)}{\sigma^2}}, 
\end{align}
is an equilibrium stopping time. The corresponding equilibrium value function is
\begin{align*}
J_{\lambda}(x)= 
\frac{1}{\left(
\sqrt{\frac{-(2\mu+\sigma^2)}{\sigma^2}}+1
\right)^2}x^2.
\end{align*}
\end{thm} 
			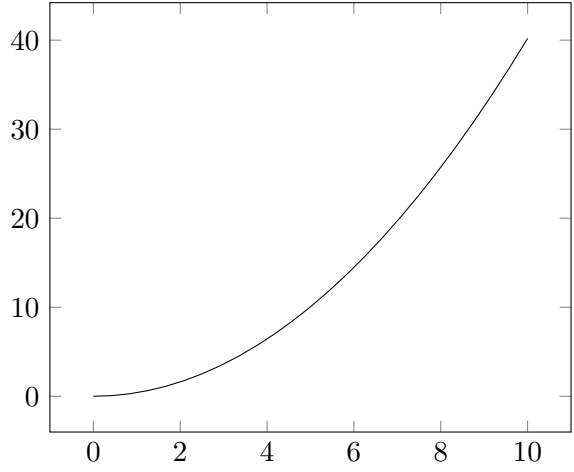
\begin{figure}[ht]
	\begin{center}
	\begin{tikzpicture}
	\begin{axis} 
	\addplot[domain=0:10] {0.401923788647*x^2}; 
	\end{axis} 
	\end{tikzpicture} 
		\end{center}
			\caption{The equilibrium value function $x\mapsto J_{\lambda}(x)$, where $\lambda$ is given by \eqref{eq-lambda}, for the parameters $\mu=-0.1$ and $\sigma^2=0.15$. In this case $\lambda \approx 0.0577$.}
\end{figure}
\begin{rem}  From the formula for the variance of the log-normal $X_t$  it follows that $\lim_{t\rightarrow \infty}\mbox{Var}_x(X_t)=0$ for any $x\in \state $ if \eqref{parameter-cond} holds, whereas $\lim_{t\rightarrow \infty}\mbox{Var}_x(X_t)=\infty$ for any $x\in \state $ if \eqref{parameter-cond} does not hold. Hence, we only consider the case when \eqref{parameter-cond} holds. We remark that  condition \eqref{parameter-cond} is also used in \cite{pedersen2011explicit}. 
\end{rem}
\begin{rem} \label{new-var-rem} For our variance problem holds, as expected, that an equilibrium stopping time cannot recommend immediate stopping at any $x$. To see this first verify that \eqref{mainthmcond2} is violated for $x\in int(C^c)$. Moreover, if $C^c$ has no interior then \eqref{mainthmcond5} is violated at every $x\in \partial C$; which can be shown using arguments similar to those in the proof of Theorem \ref{mean-varthm} below and in particular \eqref{prob-tau-b} and \eqref{prob-tau-c-d}. 
\end{rem}

\begin{proof} We remark that it follows from the calculations below that $\tau^\lambda$ is admissible. Using that $X_t=xe^{\left(\mu-\frac{1}{2}\sigma^2\right)t+\sigma W_t}$ is log-normal and conditioning on the exponentially distributed stopping time $\tau^\lambda$ we directly obtain
\begin{align}
\psi_\lambda(x)= \E_x(X_{\tau^\lambda}) = \frac{\lambda}{\lambda-\mu}x \enskip \mbox{ and }\enskip 
\phi_\lambda(x)   = \E_x(X_{\tau^\lambda}^2) = \frac{\lambda}{\lambda-2\mu-\sigma^2}x^2.
\end{align}
Here we relied on the denominators being positive, which follows directly from $\lambda>0$ and $\mu<0$, and $\lambda>0$ and \eqref{parameter-cond} respectively; where \eqref{parameter-cond} implied that $\mu<0$ and $\lambda>0$.  
It follows that 
\begin{align}
\psi^2_\lambda(x)  = \frac{\lambda^2}{(\lambda-\mu)^2}x^2, \enskip
(\psi_\lambda(x)-x)^2 =  \frac{\mu^2}{(\lambda-\mu)^2}x^2.
\end{align}
Using \eqref{var-help-eq} and \eqref{eq-lambda} we thus obtain, for any fixed $x \in E$,
\begin{align}
\frac{J_\lambda(x) - (\psi_\lambda(x)-x)^2}{x^2}
& = \frac{\phi_\lambda(x)- \psi^2_\lambda(x) - (\psi_\lambda(x)-x)^2}{x^2}\\
& = \frac{2\mu^3+\mu^2\sigma^2+\sigma^2\lambda^2}{(\lambda-2\mu-\sigma^2)(\lambda-\mu)^2}.
\end{align}
From Theorem \ref{thm:eq-condition} it therefore follows that $\tau^\lambda$ is an equilibrium stopping time when $2\mu^3+\mu^2\sigma^2+\sigma^2\lambda^2 = 0$, i.e. when $\lambda$ satisfies
$\lambda^2 = \frac{-\mu^2(2\mu+\sigma^2)}{\sigma^2}$. This proves the first assertion. Using the calculations above, $\mu < 0$ and Theorem \ref{thm:eq-condition} it is easy to find the equilibrium value function. 
\end{proof}
\begin{rem}
	In \cite{gad2015variance}, the results from \cite{pedersen2011explicit} on the pre-commitment version of the variance stopping problem are generalized to underlying geometric L\'evy processes. In this paper, we have decided to developed the theory only for underlying diffusion processes to avoid certain technical difficulties. Therefore, applying our time-consistent approach to underlying jump processes would need some further work that we do not carry out here. We, nonetheless, want to mention that obtaining equilibrium conditions of the form \eqref{equi-char} for the variance problem for underlying geometric L\'evy processes of the form\ $X_t=X_0e^{L_t},$ $L$ a L\'evy process, can also be obtained. It is then interesting to note that considering $\tau^\lambda$ for a constant $\lambda>0$ yields -- under suitable integrability conditions -- that
	\begin{align*}
	\psi_\lambda(x)&= \E_1(X_{\tau^\lambda})x = a_\lambda x, \enskip a_\lambda=\frac{\lambda}{\lambda-\Psi_L(1)} \enskip \mbox{ and }\\
	\phi_\lambda(x)&= \E_1(X_{\tau^\lambda}^2)x^2 = b_\lambda x^2, \enskip b_\lambda=\frac{\lambda}{\lambda-\Psi_{2L}(1)},
	\end{align*}
	where $\Psi$ denotes the Laplace exponent. Hence, a similar calculation as in the previous proof yields both a formula for $\lambda$ and the corresponding equilibrium value function also in this case. 
\end{rem}

\subsection{A mean-variance stopping problem} \label{applications:mean-var}
Mean-variance optimization is one of the classical problems in financial economics. It was first studied in the context of optimal portfolio allocation in the seminal paper \cite{markowitz1952portfolio}. A vast number of papers on the topic have since then been published. For short surveys and economic motivation of mean-variance problems we refer to \cite{bjork2014mean,pedersen2016optimal} and the references therein. The mean-variance stopping problem corresponds to the time-inconsistent problem of trying to maximize
\[\E_x(X_\tau) - \gamma \mbox{Var}_x(X_\tau), \mbox{ with $\gamma>0$.}\]
Here $\gamma$ is a given constant representing risk-aversion.  
In \cite{pedersen2016optimal} a mean-variance stopping problem for a geometric Brownian motion is studied using the dynamic optimality approach and the pre-commitment approach. In \cite{bayraktar2018} a mean-variance stopping problem for a general discrete time Markov chain is studied, see also Section \ref{equilibrium-disc}. In \cite{bjork2014mean} a mean-variance control problem is studied using the general game-theoretic framework for time-inconsistent stochastic control of \cite{tomas-continFORTH}.

The mean-variance stopping problem is given by $f(x):=-\gamma x^2, g(x):=x+\gamma x^2$ and $h(x):=x$. To see this note that
\begin{align}
J_{\tau}(x) &= \phi_\tau(x) + g(\psi_\tau(x))\\ 
&= -\gamma\E_x(X_\tau^2) + \E_x(X_\tau) + \gamma \E_x^2(X_\tau) \\
&= \E_x(X_\tau) - \gamma \mbox{Var}_x(X_\tau).
\end{align} 
We consider a positive state process $X$. In this case Assumption \ref{ass:general-lemmas} is satisfied. Note that $g'(h(x))=1+ 2 \gamma x$, $g'(\psi_{\lambda,C}(x)) = 1 + 2\gamma \psi_{\lambda,C}(x)$, and 
 $J_{\lambda,C}(x) =   \phi_{\lambda,C}(x) + \psi_{\lambda,C}(x)+\gamma \psi^2_{\lambda,C}(x)$. Therefore,  simple calculations give
\begin{align} 
& f(x) -\phi_{\lambda,C}(x) + g'(\psi_{\lambda,C}(x))\left(h(x)-\psi_{\lambda,C}(x)\right)\\
& = -\gamma x^2 -\phi_{\lambda,C}(x) + (1 + 2\gamma \psi_{\lambda,C}(x))\left(x-\psi_{\lambda,C}(x)\right)\\
& = x -J_{\lambda,C}(x)-\gamma(\psi_{\lambda,C}(x)-x)^2.
\end{align}
It follows that conditions \eqref{mainthmcond3} and \eqref{mainthmcond4} can be written as 
\begin{align}
J_{\lambda,C}(x) & = x-\gamma(\psi_{\lambda,C}(x)-x)^2,\enskip\mbox{ for   $x \in C$ with $\lambda(x)>0$}, \label{mean-vareq1}\\
J_{\lambda,C}(x) &\geq x-\gamma(\psi_{\lambda,C}(x)-x)^2,\enskip\mbox{ for   $x \in C$ with $\lambda(x)=0$}.\label{mean-vareq2}
\end{align}
Using that $f(x)+g(h(x))=x$ we write condition \eqref{mainthmcond1} as,
\begin{align}
J_{\lambda,C}(x) \geq x,\enskip\mbox{ for $x \in C$}. \label{mean-vareq3}
\end{align}
Let us again consider the geometric Brownian motion. In the typical case it is reasonable to suppose that $J_{\lambda,C}(x)-x>0$ for $x \in C$ and in this case we note, using Lemma \ref{prop-calc-eq-cond1}, that if $\tau^{\eta,D}\in \mathcal{N}$ with $\eta=0$, then, for $x\in C \cap D$, 
\begin{align}
& \lim_{h\searrow 0}\frac{J_{\tau^{\lambda,C}}(x)-J_{\tau^{\lambda,C}\diamond \tau^{\eta,D}(h)}(x)}{\E_x(\tau_h)}\\ 
& \enskip = \lambda(x)\{f(x)-\phi_{\lambda,C}(x)+ g'(\psi_{\lambda,C}(x))\left(h(x)-\psi_{\lambda,C}(x)\right)\}\\
& \enskip= \lambda(x)\{x -J_{\lambda,C}(x)-\gamma(\psi_{\lambda,C}(x)-x)^2\}\\
& \enskip<0.\label{mean-var:nomix}
\end{align}
Consequently we make the ansatz $\lambda(x)=0$  for $x\in C$. Specifically, we make the ansatz that $\tau^C$ for $C=(0,b)$ is an equilibrium stopping time for some $b$ to be determined. We start by noting that if $\tau^C$ satisfies \eqref{mean-vareq3} then condition \eqref{mainthmcond1} and condition \eqref{mainthmcond4} are satisfied, and condition \eqref{mainthmcond3} is irrelevant (since the ansatz is $\lambda=0$ on $C$). Hence, if we can find a set $C=(0,b)$ such that \eqref{mean-vareq3},  \eqref{mainthmcond2} and \eqref{mainthmcond5} are satisfied then $\tau^C$ is an equilibrium strategy.

\begin{thm} \label{mean-varthm} Let $X$ be given by
\[dX_t = \mu X_t dt + \sigma X_t dW_t, \enskip \mbox{ where $\sigma^2>0$.}\]
If   $\mu \in  (0,\sigma^2/4]$, then $\hat\tau=\inf\{t\geq0: X_t \geq b\}$ with $b=\frac{\xi}{\gamma(1-\xi)}$, where
 $\xi:=\frac{2\mu}{\sigma^2}$, is an equilibrium stopping time and the corresponding equilibrium value function is,
\[J_{\hat\tau}(x)=\begin{cases}  
	x,&\;x \geq b,\\
x^{1-\xi}(b^\xi-\gamma b^{1+\xi}) + \gamma b^{2\xi} x^{2-2\xi},&\; x<b.
	\end{cases}  
	\]
If $\mu\in(\sigma^2/4,\sigma^2/2)$, then no equilibrium stopping time exists. 
\end{thm}
\begin{figure}[ht]
	\begin{center}
	\begin{tikzpicture} 
	\begin{axis} 
		  \addplot[style=dashed,domain=0:0.5] {x};  
			\addplot[style=solid,domain=0:0.4105572]{
x^(1-0.3111111111)*(0.41055718^(0.3111111111)-1.1*0.41055718^(1+0.3111111111))+ 1.1*0.41055718^(2*0.3111111111)*x^(2-2*0.3111111111)
						};  
      \addplot[style=solid,domain=0.4105572:0.5]{x};
	\end{axis} 
	\end{tikzpicture} 
		\end{center}
			\caption{The equilibrium value function $x\mapsto J_{\hat \tau}(x)$ (solid) and $x\mapsto f(x)+g(x) = x$ (dashed) in the GBM case with parameters $\mu=0.07, \sigma^2=0.45$  and $\gamma=1.1$ (in this case $b \approx 0.4106$).} 
\end{figure}
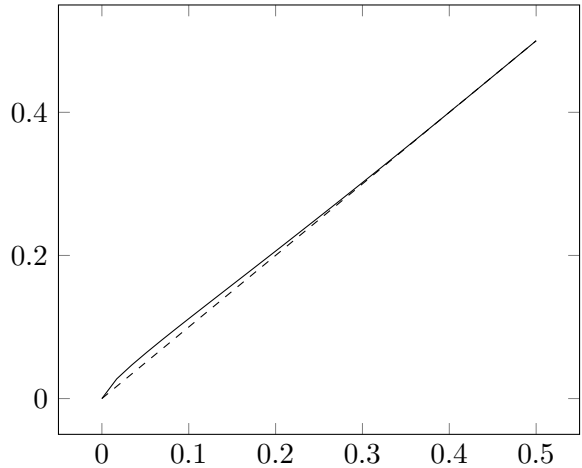
\begin{rem} \label{mean-var-rem} If $\mu\leq0$ then $X$ is a supermartingale (with a last element) and it follows directly from $J_{\tau}(x) = \E_x(X_\tau) - \gamma \mbox{Var}_x(X_\tau)$, Definition \ref{def:equ_stop_time} and the optional sampling theorem that it is an equilibrium strategy to always stop immediately. 
If $\mu\geq \frac{\sigma^2}{2}$ then $\tau^b:=\inf\{t\geq0: X_t \geq b\}<\infty$ a.s. for any initial state $x \leq b$ for each $b\in \state$ and $J_{\tau^b}(x) = \E_x(X_{\tau^b}) - \gamma\left(\E_x(X_{\tau^b}^2)-\E_x^2(X_{\tau^b})\right) = b - \gamma(b^2-b^2) = b$ can thus become arbitrarily large. 
\end{rem}

\begin{rem} A mean-variance optimal stopping problem for a GBM is studied in \cite{pedersen2016optimal}. There it is shown that the stopping time $\hat \tau$ in Theorem \ref{mean-varthm} is \emph{dynamically optimal} when $\mu \in  (0,\sigma^2/2)$, see \cite[Theorem 3]{pedersen2016optimal}. It is also argued that this stopping time is a subgame perfect Nash equilibrium when $\mu \in (0, \sigma^2/4]$, see \cite[Sec. 4]{pedersen2016optimal}, which is in line with our findings in Theorem \ref{mean-varthm}. 
\end{rem}
\begin{proof}  
We remark that it follows from the calculations below that $\hat \tau$ is admissible. A stopping time is, according to Theorem \ref{main-thm-imm-stp}, an equilibrium stopping time if and only if conditions \eqref{mainthmcond1}---\eqref{mainthmcond5} are satisfied. Note that we do not have to check \eqref{mainthmcond3} since $\hat\tau$ has no Cox process component, which corresponds to $\lambda(x)=0$ for each $x$. Recall that if \eqref{mean-vareq3} is satisfied then \eqref{mainthmcond1} and \eqref{mainthmcond4} are also satisfied. Note that \eqref{mainthmcond2} can in this case be written as
\begin{align}
A_Xf(x) + g'(h(x))A_X h(x)  
= x(-\gamma\sigma^2x + \mu) 
\leq 0, \enskip \mbox{ for } x\in \mbox{int} (C^c).\label{3.5.meanvar}
\end{align}
It follows that if we can verify \eqref{mean-vareq3}, \eqref{3.5.meanvar} and \eqref{mainthmcond5} for $\hat\tau$ then we are done. Let us now consider the candidate equilibrium stopping time $\tau^b:=\inf\{t\geq0: X_t \geq b\}$ and use the smooth fit condition to see that necessarily $b=\frac{\xi}{\gamma(1-\xi)}$. Recall, from standard theory, that for any $b$, 
\begin{align}
\mathbb P_x(\tau^b<\infty) = b^{\xi-1}x^{1-\xi},  \enskip \mbox{ for } x \leq b. \label{prob-tau-b}
\end{align}
Since $X_t\rightarrow 0$ a.s. as $t\rightarrow \infty$, it hence holds, for any $x\leq b$, that
\begin{align}
\psi_{\tau^b}(x)&= \E_x(X_{\tau^b})\\
&= \E_x(X_{\tau^b}I_{\{\tau^b<\infty\}}) + \E_x(X_{\tau^b}I_{\{\tau^b=\infty\}}) \\
& = b^\xi x^{1-\xi}.
\end{align}
Similarly, $\E_x(X^2_{\tau^b}) = b^{1+\xi} x^{1-\xi}.$
Hence, for $x\leq b$, 
\begin{align}
J_{\tau^b}(x) & =\E_x(X_{\tau^b}) - \gamma \left(\E_x(X_\tau^2)-\E_x^2(X_\tau)\right)\\
& = x^{1-\xi}(b^\xi-\gamma b^{1+\xi}) + \gamma b^{2\xi} x^{2-2\xi}.
\end{align}
It is easy to verify that $J_{\tau^b}(b) = b$, for any $b$, and hence the function
\[J_{\tau^b}(x)=\begin{cases}  
	x,&\;x \geq b,\\
x^{1-\xi}(b^\xi-\gamma b^{1+\xi}) + \gamma b^{2\xi} x^{2-2\xi},&\; x<b,
	\end{cases}  
	\]
	is continuous. Note that
\[J'_{\tau^b}(x)=\begin{cases}  
	1,&\;x > b,\\
(1-\xi)x^{-\xi}(b^\xi-\gamma b^{1+\xi}) + (2-2\xi)\gamma b^{2\xi} x^{1-2\xi},&\; x<b,
	\end{cases}  
	\]
where the lower part is, for $x=b$, equal to:
\begin{align}
(1-\xi)b^{-\xi}(b^\xi-\gamma b^{1+\xi}) + (2-2\xi)\gamma b^{2\xi} b^{1-2\xi}&= (1-\xi)(1-\gamma b) + 2(1-\xi)\gamma b\\
&= (1-\xi)(1+ \gamma b).
\end{align}
In order for the smooth fit condition (Theorem \ref{smoothfit-thm}) to be satisfied we need that $J'_{\tau^b}(b)$ is equal to $f'(x)+g'(h(x))h'(x)=1$. We thus need that
$(1-\xi)(1+ \gamma b)=1$. Hence, the only possible $b$ is given by
\begin{align}
b = \frac{1}{\gamma} \left(\frac{1}{(1-\xi)}-1\right) =\frac{\xi}{\gamma(1-\xi)}.
\end{align}
It is easily verified that \eqref{3.5.meanvar} holds when $b=\frac{\xi}{\gamma(1-\xi)}$, using that $\mu \in  (0,\sigma^2/4]$  (i.e. $\xi \in (0,1/2]$). From the explicit form of $J_{\tau^b}(x)$ above it follows that \eqref{mean-vareq3} is satisfied exactly when
\begin{align} \label{mea-varh22}
x^{1-\xi}(b^\xi-\gamma b^{1+\xi}) + \gamma b^{2\xi} x^{2-2\xi} -x\geq  0, \enskip \mbox{ for } x < b.
\end{align}
It is straightforward to show that this inequality is satisfied,  using that $\xi \in (0,1/2]$ and $b=\frac{\xi}{\gamma(1-\xi)}$, and thereby verifying \eqref{mean-vareq3}. The only thing we have left is to verify \eqref{mainthmcond5}, which we will do using Theorem \ref{sufficient-cond-thm}. From the calculations above follows that
\begin{align}
\phi_{\tau^b}(x)&=\E_x(-\gamma X^2_{\tau^b}) = \begin{cases} 
	-\gamma x^2 ,&\;x\geq b,\\
-\gamma b^{1+\xi} x^{1-\xi},&\; x<b,
	\end{cases}\\
	\psi_{\tau^b}(x)&=\E_x(X_{\tau^b})=\begin{cases} 
	x ,&\;x \geq b,\\
b^\xi x^{1-\xi},&\; x<b.
	\end{cases} 
\end{align}	
Let us drop the subscript $\tau^b$. It follows that
\begin{align}
\phi'(x) &= \begin{cases} 
 -2\gamma x,				&\;x> b,\\
 -(1-\xi)\gamma b^{1+\xi} x^{-\xi}, &\; x< b,\\
		\end{cases}\\
\phi''(x) &= \begin{cases} 
 -2\gamma ,				&\;x> b,\\
	\xi(1-\xi)\gamma b^{1+\xi} x^{-\xi-1} ,														 			&\; x< b.\\
		\end{cases}
	\end{align}
Thus, 
\begin{align} A_X\phi(x) & = \mu x\phi'(x) + \frac{1}{2}\sigma^2x^2\phi''(x)\\
& = \begin{cases} 
-2\mu\gamma x^2	- \frac{1}{2}\sigma^2x^2			2\gamma						, &\;x> b,\\
	-\mu x (1-\xi)\gamma b^{1+\xi} x^{-\xi} + \frac{1}{2}\sigma^2x^2 \xi(1-\xi)\gamma b^{1+\xi} x^{-\xi-1}	 , &\; x< b,\\
	\end{cases}\\ 
	& = \begin{cases} 
-2\gamma x^2(\mu + \frac{1}{2}\sigma^2) 						, &\;x> b,\\
0, &\; x< b,\\
	\end{cases} 
	\end{align}
where we in the last equality used that $\xi = 2\mu/\sigma^2$. 
Similarly, 
\begin{align}\psi'(x) & = \begin{cases} 
 1,				&\;x> b,\\
(1-\xi)b^\xi x^{-\xi}	 ,														 			&\; x< b,\\
		\end{cases}
	\\
	\psi''(x) & = \begin{cases} 
	0 ,				&\;x> b,\\
-\xi(1-\xi)b^\xi x^{-\xi-1} ,														 			&\; x< b,\\
		\end{cases}\\
A_X\psi(x) & = 
	\begin{cases} 
 \mu x   , &\;x> b,\\
 0, &\; x< b.\\
	\end{cases} 
	\end{align}	
Note that $g'(\psi(b)) = 1+ 2\gamma\psi(b) = 1+ 2\gamma b$. Thus, 
\begin{align} 
A_X\phi(x) +  g'(\psi(b))A_X\psi(x)
= \begin{cases} 
 x\{\mu+ 2\gamma \mu(b-x) -\gamma \sigma^2x \}  , &\;x> b,\\
 0, &\; x< b.\\
	\end{cases} \label{meanvarxtxtx}
	\end{align} 		
It is easily checked that $\phi$ and $\psi$ are a $\mathcal{C}^2$ everywhere except at $x = b$ and 
that $\phi(\cdot) + g'(\psi(b)) \psi(\cdot)$ is $\mathcal{C}^1$ everywhere.  Hence, we may use Theorem \ref{sufficient-cond-thm}. Let us verify that  \eqref{sufficient-cond-thm:1} holds:

Trivially, $g''(b) = 2\gamma$. For the GBM it holds that $\sigma^2(x) = x^2 \sigma^2$. Moreover, 
$\frac{\xi^2}{2}\sigma^2 = \xi\mu$, $\gamma b = \frac{\xi}{1-\xi}$ and $-1+\sigma^2\xi/\mu=1$. Using these findings, including \eqref{meanvarxtxtx}, we obtain 
\begin{align} 
&  A_X\phi(b+) + g'( \psi(b))A_X\psi(b+)  +  A_X\phi(b-) + g'( \psi(b))A_X\psi(b-) \\
& \enskip\enskip + g''( \psi(b))\left(\frac{\psi'(b+) - \psi'(b-)}{2}\right)^2 \sigma^2(b)\\
&\enskip= b \left(\mu-\gamma \sigma^2 b - \gamma b\frac{\xi^2}{2}\sigma^2\right)\\
&\enskip = -b\mu\left(\frac{\xi-1 + \sigma^2\xi/\mu - \xi^2}{1-\xi}\right)\\
&\enskip= -b\mu\frac{1+\xi-\xi^2}{1-\xi} \leq 0,
\end{align}
where the inequality follows from $\xi \in (0,1/2]$. This means that \eqref{sufficient-cond-thm:1} holds, which, by Theorem \ref{sufficient-cond-thm}, implies that condition \eqref{mainthmcond5} holds and the first statement of the theorem follows. 

Let  us consider $\mu\in(\sigma^2/4,\sigma^2/2)$. From \eqref{mean-vareq3} and the calculations in \eqref{mean-var:nomix} follows that an equilibrium must satisfy either 
$\lambda(x)=0$ 
or $\psi_{\lambda,C}(x)=x$ and $J_{\lambda,C}(x)=x$, for each $x\in C$.
But if $\psi_{\lambda,C}(x)=\E_x(X_\tau)=x$ and $J_{\lambda,C}(x) = \E_x(X_\tau) - \gamma \mbox{Var}_x(X_\tau)=x$ then $\mbox{Var}_x(X_\tau)=0$. Now, the only way $\mbox{Var}_x(X_\tau)=0$ holds for $x\in C$ is that $\tau$ is the threshold time for some constant 
$0<c<x$, by basic properties of the GBM in the present case. But this implies $J_{\lambda,C}(x) = \E_x(X_\tau) - \gamma \mbox{Var}_x(X_\tau)=c<x$ which violates \eqref{mean-vareq3}. Hence, for equilibria it must hold for each $x\in C$ that $\lambda(x)=0$. 
To prove the second statement of the theorem it is thus enough to prove that $\tau^C$ cannot be an equilibrium stopping time for an arbitrary continuation set $C$. Since $C$ is open, cf. Definition \ref{def:mix_strat}, follows that $C$ must be described by either of the cases below. 
We conclude the proof by showing that none of these cases allow for equilibria.

\emph{Case 1: $C=\emptyset$.} In this case there exists an $x \in \mbox{int} (C^c)$ such that the inequality in \eqref{3.5.meanvar} does not hold, cf. $\mu,\gamma,\sigma^2>0$. Hence, $C=\emptyset$ is not an equilibrium.

\emph{Case 2: $C=(0,\infty)$.} In this case $J_C(x)=0$ and hence \eqref{mean-vareq3} does not hold. Hence, $C=(0,\infty)$ is not an equilibrium.

\emph{Case 3: $(0,c)\subset C$ for some constant $0<c<\infty, c\in \partial C$.} 
This corresponds to the type of stopping time investigated in the first part of this proof. Using the arguments before and after \eqref{mea-varh22} we find that for a stopping time of this type it is required that $\mu\in (0,\sigma^2/4]$ in order for \eqref{mean-vareq3} to hold. Hence, no equilibrium exists for this case.

\emph{Case 4: $(c,d)\subset C$ for some constants $0<c<d< \infty, c,d \in \partial C$.} Consider an $x \in (c,d)$. From basic properties of the GBM follows,
\begin{align}
\mathbb P_x(\tau^{c,d}=d)  = \frac{x^{1-\xi}-c^{1-\xi}}{d^{1-\xi}-c^{1-\xi}},  \enskip x \in [c,d]. \label{prob-tau-c-d}
\end{align}
Simple calculations give,
\[\phi_{\tau^{c,d}}(x)=-\gamma \E_x(X^2_{\tau^{c,d}}) = \begin{cases} 
-\gamma \frac{(d^2-c^2)x^{1-\xi}+ d^{1-\xi}c^2 -c^{1-\xi}d^2 }{d^{1-\xi}-c^{1-\xi}} ,&\; x \in (c,d)\\
	-\gamma x^2 ,&\; x \in \{c,d\},
	\end{cases}\]
\[\psi_{\tau^{c,d}}(x)= \E_x(X_{\tau^{c,d}}) = \begin{cases} 
\frac{(d-c)x^{1-\xi}+ d^{1-\xi}c -c^{1-\xi}d }{d^{1-\xi}-c^{1-\xi}} ,&\; x \in (c,d)\\
x  ,&\; x \in \{c,d\}.
	\end{cases}\]
For $x\in (c,d)$ we can thus determine constants $c_i$, such that,
\begin{align} 
\psi_{\tau^{c,d}}(x) = c_1x^{1-\xi} + c_2, \enskip \phi_{\tau^{c,d}}(x) = c_3x^{1-\xi} + c_4.
\end{align}
Using the same notation we find for the function $M(x):= J_{{c,d}}(x)-x$ and $x\in (c,d)$ that,
\begin{align}
M(x) & =\phi_{\tau^{c,d}}(x) + \psi_{\tau^{c,d}}(x) + \gamma \psi_{\tau^{c,d}}^2(x)-x\\
& = c_5x^{1-\xi} + c_6+ c_7x^{2-2\xi}-x,\\
M'(x) &= c_8x^{-\xi} + c_9x^{1-2\xi}-1,\\
M''(x) &= x^{-2\xi}(c_{10}x^{\xi-1} +c_{11}).
\end{align}
Now suppose $\tau^{c,d}$ is an equilibrium stopping time, then, by smooth fit, $M'(c)=M'(d)=0$ 
(sufficient differentiability around the points $c$ and $d$ for the use of Theorem \ref{smoothfit-thm} is easily seen to be fulfilled).
Moreover, it is directly seen that $M''(x)$ has at most one zero in $(c,d)$; and that in order for $M'(c)=M'(d)=0$ to be true it must indeed have a zero in $(c,d)$. From this follows that $M$ is strictly increasing or strictly decreasing on $(c,d)$ and hence that $M(c)$ and $M(d)$ cannot both be equal to zero, which, by definition of $M$, means that $\tau^{c,d}$ cannot be an equilibrium stopping time, and we have reached a contradiction. Hence, no equilibrium exists for this case.

\emph{Case 5: $(c,\infty)\subset C$ for some constant $0<c \in \partial C$.} In this case it is easy to see that 
either $C^c = (0,c]$ or $C$ contains a bounded interval. For the first alternative follows directly that \eqref{3.5.meanvar} is violated. The second alternative is covered by Case 4. Hence, no equilibrium exists for this case. 
\end{proof}

\subsection{An example with two equilibria}\label{subsec:example-2-equi}
Here we present an example with two different equilibrium value functions, implying unique equilibria cannot generally be expected. 
Suppose $X$ is a Wiener process. Let 
	$f(x)= \frac{x^6}{9} - \frac{x^4}{3}$, 
	$g(x)=  x^2 - \frac{5x^3}{9}$ and 
	$h(x)= x^2$.  
Define the stopping time $\hat\tau = 0$, which corresponds to $C=\emptyset$ and any intensity function $\lambda$. Condition \eqref{mainthmcond2} of Theorem \ref{main-thm-imm-stp} is then directly verified and the other conditions are in this case irrelevant. Hence, $\hat\tau = 0$ is an equilibrium stopping time and the corresponding equilibrium value function is 
$\hat J(x) = f(x) + g(h(x))$. 
Now define $\tilde\tau$ as the stopping time corresponding to $C=(-1,1)$ and $\lambda = 0$. With obvious notation and basic properties of the Wiener process we find, 
\begin{align} \label{some-obs-in-ex}
\tilde\phi(x) =\frac{-2}{9}, \enskip \tilde\psi(x)  = 1, \enskip \tilde J(x) = \frac{2}{9}, \mbox{ for } x \in C.
\end{align}
Using these observations and tedious calculations it can be verified that  
\eqref{mainthmcond1} and \eqref{mainthmcond4} hold, whereas \eqref{mainthmcond2} is verified as above and \eqref{mainthmcond3} is irrelevant. Condition \eqref{mainthmcond5} can be verified by explicit calculation of $a(x,h)$, for $x=-1,1$; where the main observation is that e.g.
$\E_1\left(\tilde\phi(X_{\tau_h})\right) = \frac{1}{2}\frac{-2}{9} + \frac{1}{2}f(1+h)$, which is seen using \eqref{some-obs-in-ex} and that $X$ is a Wiener process.

\section{Appendix}
\begin{lem} \label{main-lemma2} For any $\tau^{\lambda,C} \in \mathcal{N}$ and $x\in C$,
\begin{align}
A_X \phi_{\lambda,C}(x) 
& = \lim_{h\searrow 0}\frac{\phi_{{\tau^{\lambda,C}}\circ \theta_{\tau_h}+\tau_h}(x)-\phi_{\lambda,C}(x)}{\E_x(\tau_h)} \\
& = \lambda(x)(\phi_{\lambda,C}(x)-f(x)). \label{main-lemma2:eq} 
\end{align}
\end{lem}
\begin{proof} Using arguments similar to those we used to arrive at \eqref{dev-strat-prop} and the strong Markov property  we obtain
\begin{align*}
\phi_{\tau^{\lambda,C} \diamond \tau^{\lambda,C}(h)}(x) 
&= \E_x\left(f\left(X_{\tau^{\lambda,C} \diamond \tau^{\lambda,C}(h)}\right)\right)\\
&= \E_x\left(I_{\{\tau^\lambda \leq \tau_h\}}f\left(X_{\tau^{\lambda,C}}\right) +  I_{\{\tau^\lambda > \tau_h\}}f\left(X_{{\tau^{\lambda,C}}\circ \theta_{\tau_h}+\tau_h}\right)\right)\\
& = \E_x\left(f(X_{\tau^{\lambda,C}})\right) \\ 
&= \phi_{\lambda,C}(x),
\end{align*}
for $0<h\leq \bar{h}$, for some $\bar{h}>0$. 
This implies that the second equality in \eqref{main-lemma2:eq} follows from Lemma \ref{main-lemma}. Now use the strong Markov property to see that
\begin{align}
\phi_{{\tau^{\lambda,C}}\circ \theta_{\tau_h}+\tau_h}(x) 
& = \E_x\left(f\left(X_{{\tau^{\lambda,C}}\circ \theta_{\tau_h}+\tau_h}\right)\right)\\
&= \E_x\left(\E_x\left(f\left(X_{{\tau^{\lambda,C}}\circ \theta_{\tau_h}+\tau_h}\right)|\mathcal{F}_{\tau_h}\right) \right)\\
&= \E_x\left(\phi_{\lambda,C}(X_{\tau_h})  \right).
\end{align}
Hence, the first equality in \eqref{main-lemma2:eq} follows from the definition of the characteristic operator $A_X$. 
\end{proof}
\begin{lem} \label{prop-calc-eq-cond1} For any $\tau^{\lambda,C},\tau^{\eta,D} \in \mathcal{N}$ and $x \in C \cap D$,
\begin{align}
& \lim_{h\searrow 0}\frac{J_{\lambda,C}(x)-J_{\tau^{\lambda,C}\diamond \tau^{\eta,D}(h)}(x)}{\E_x(\tau_h)}\\ 
& \quad = (\lambda(x)-\eta(x))\{f(x)-\phi_{\lambda,C}(x)+ g'(\psi_{\lambda,C}(x))\left(h(x)-\psi_{\lambda,C}(x)\right)\}.
\end{align}
\end{lem} 
\begin{proof}
Use the same argument as in the proof of Lemma \ref{main-lemma2} to obtain
\begin{align}
& J_{\tau^{\lambda,C}}(x) -  J_{\tau^{\lambda,C}\diamond \tau^{\eta,D}(h)}(x) \\
& = J_{{\tau^{\lambda,C}\diamond \tau^{\lambda,C}(h)}}(x)-J_{{\tau^{\lambda,C}}\circ \theta_{\tau_h}+\tau_h}(x)
- (J_{{\tau^{\lambda,C}\diamond \tau^{\eta,D}(h)}}(x)-J_{{\tau^{\lambda,C}}\circ \theta_{\tau_h}+\tau_h}(x))
\label{prop-calc-eq-cond1-1}.
\end{align}
The second part of \eqref{prop-calc-eq-cond1-1} can, by definition, be written as
\begin{align}
& J_{{\tau^{\lambda,C}\diamond \tau^{\eta,D}(h)}}(x)-J_{{\tau^{\lambda,C}}\circ \theta_{\tau_h}+\tau_h}(x)\\  
& \enskip = \phi_{{\tau^{\lambda,C}\diamond \tau^{\eta,D}(h)}}(x) - \phi_{{\tau^{\lambda,C}}\circ \theta_{\tau_h}+\tau_h}(x) + g(\psi_{{\tau^{\lambda,C}\diamond \tau^{\eta,D}(h)}}(x))
- g(\psi_{{\tau^{\lambda,C}}\circ \theta_{\tau_h}+\tau_h}(x)).
\end{align}
From Lemma \ref{main-lemma}  it follows that
\begin{align*}
\lim_{h\searrow 0}
\frac{\phi_{{\tau^{\lambda,C}\diamond \tau^{\eta,D}(h)}}(x) - \phi_{{\tau^{\lambda,C}}\circ \theta_{\tau_h}+\tau_h}(x)}{\E_x(\tau_h)} 
=\eta(x)(f(x)-\phi_{\lambda,C}(x)).
\end{align*}
Use the same arguments as for \eqref{proof-ref} to obtain (here $\eta_t := \eta(X_t)$)
\begin{align*}
& \psi_{\tau^{\lambda,C} \diamond \tau^{\eta,D}(h)}(x)  \\
& = \psi_{{\tau^{\lambda,C}}\circ \theta_{\tau_h}+\tau_h}(x) +
 \E_x\left(\int_0^{\tau_h}\eta_t e^{-\int_0^t\eta_s ds}\left(h(X_t)-\psi_{{\lambda,C}}\left(X_{\tau_h}\right)\right)dt\right).
\end{align*}
Using standard Taylor expansion we thus obtain
\begin{align*}
& g(\psi_{{\tau^{\lambda,C}\diamond \tau^{\eta,D}(h)}}(x))- g(\psi_{{\tau^{\lambda,C}}\circ \theta_{\tau_h}+\tau_h}(x)) \\
& = g\left(\psi_{{\tau^{\lambda,C}}\circ \theta_{\tau_h}+\tau_h}(x) +
 \E_x\left(\int_0^{\tau_h}{\eta_t} e^{-\int_0^t{\eta_s} ds}\left(h(X_t)
-\psi_{{\lambda,C}}\left(X_{\tau_h}\right)\right)dt\right)\right) \\
& \enskip - g(\psi_{{\tau^{\lambda,C}}\circ \theta_{\tau_h}+\tau_h}(x))\\
& =  g'\left(\psi_{{\lambda,C}\circ \theta_{\tau_h}+\tau_h}(x) \right)\E_x\left(\int_0^{\tau_h}{\eta_t} e^{-\int_0^t{\eta_s} ds}\left(
h(X_t)-\psi_{{\lambda,C}}\left(X_{\tau_h}\right)\right)dt\right) + o(\E_x(\tau_h)).
\end{align*}
Use the equality above and $\psi_{{\tau^{\lambda,C}}\circ \theta_{\tau_h}+\tau_h}(x)= \E_x\left(\psi_{\lambda,C}(X_{\tau_h})\right)$ to obtain  
\begin{align*}
\lim_{h\searrow 0}
\frac{g(\psi_{{\tau^{\lambda,C}\diamond \tau^{\eta,D}(h)}}(x))- g(\psi_{{\tau^{\lambda,C}}\circ \theta_{\tau_h}+\tau_h}(x))}{\E_x(\tau_h)} 
= g'(\psi_{{\lambda,C}}(x))\eta(x)(h(x)-\psi_{\tau^{\lambda,C}}(x)).
\end{align*} 
Putting the above together gives us that the limit for the second part of \eqref{prop-calc-eq-cond1-1} satisfies
\begin{align}
& \lim_{h\searrow 0}\frac{J_{{\tau^{\lambda,C}\diamond \tau^{\eta,D}(h)}}(x)-J_{{\tau^{\lambda,C}}\circ \theta_{\tau_h}+\tau_h}(x)}{\E_x(\tau_h)}\\ 
& \quad = \eta(x)\{f(x)-\phi_{\lambda,C}(x)+ g'(\psi_{\lambda,C}(x))\left(h(x)-\psi_{\lambda,C}(x)\right)\}.
\label{prop-calc-eq-cond1-2}
\end{align}
In the same way we obtain that the limit for the first part of \eqref{prop-calc-eq-cond1-1} satisfies
\begin{align}
& \lim_{h\searrow 0}\frac{J_{{\tau^{\lambda,C}\diamond \tau^{\lambda,C}(h)}}(x)-J_{{\tau^{\lambda,C}}\circ \theta_{\tau_h}+\tau_h}(x)}{\E_x(\tau_h)}\\ 
& \quad = \lambda(x)\{f(x)-\phi_{\lambda,C}(x)+ g'(\psi_{\lambda,C}(x))\left(h(x)-\psi_{\lambda,C}(x)\right)\}.
\label{prop-calc-eq-cond1-3}
\end{align}
The result follows from \eqref{prop-calc-eq-cond1-1}, \eqref{prop-calc-eq-cond1-2} and \eqref{prop-calc-eq-cond1-3}.
\end{proof}

\begin{lem} \label{prop-calc-eq-cond2} 
For any $\tau^{\lambda,C},\tau^{\eta,D} \in \mathcal{N}$ and $x \in \mbox{int}(C^c)\cap D$,
\begin{align*}
\lim_{h\searrow 0}\frac{J_{{\lambda,C}}(x)-J_{\tau^{\lambda,C}\diamond \tau^{\eta,D}(h)}(x)}{\E_x(\tau_h)}
=  -A_Xf(x) - g'(h(x))A_Xh(x).
\end{align*}
\end{lem}

\begin{proof} Since $D$ and $\mbox{int}(C^c)$ are open it follows that there exists a constant $\bar{h}>0$ such that, for $0<h\leq \bar{h}$, 
\begin{align}
\tau^{\lambda,C} \diamond \tau^{\eta,D}(h) 
& = I_{\{\tau^{\eta,D} \leq \tau_h\}}\tau^{\eta,D} + I_{\{\tau^{\eta,D} > \tau_h\}}({\tau^{\lambda,C}}\circ \theta_{\tau_h}+\tau_h)\\
& = I_{\{\tau^{\eta} \leq \tau_h\}}\tau^{\eta} + I_{\{\tau^{\eta} > \tau_h\}}({\tau^{\lambda,C}}\circ \theta_{\tau_h}+\tau_h)\\
&= \tau^{\eta} \wedge \tau_h.
\end{align}
Since $x \in \mbox{int}(C^c)\cap D$ it follows that
\begin{align}
& J_{\tau^{\lambda,C}}(x)-J_{\tau^{\lambda,C}\diamond \tau^{\eta,D}(h)}(x)\\
& = f(x) - \phi_{\tau^{\lambda,C}\diamond \tau^{\eta,D}(h)}(x) + g(h(x))-g(\psi_{\tau^{\lambda,C}\diamond \tau^{\eta,D}(h)}(x))
\label{prop-3.4eq}.
\end{align}
Use It\^{o}'s formula to rewrite the first part of \eqref{prop-3.4eq} as 
\begin{align}
f(x) - \phi_{\tau^{\lambda,C}\diamond \tau^{\eta,D}(h)}(x)
&= f(x)-\E_x\left(f\left(X_{\tau^{\lambda,C}\diamond \tau^{\eta,D}(h)}\right)\right)\\ 
&= -\E_x\left(\int_0^{{\tau^{\eta} \wedge \tau_h}}A_Xf(X_t)dt\right).
\end{align}
It follows that  
\[\lim_{h\searrow 0}\frac{f(x) - \phi_{\tau^{\lambda,C}\diamond \tau^{\eta,D}(h)}(x)}{\E_x(\tau_h)} = -A_Xf(x).\] 
Use similar arguments and standard Taylor expansion to rewrite the second part of \eqref{prop-3.4eq}
\begin{align} 
g(h(x)) - g(\psi_{\tau^{\lambda,C}\diamond \tau^{\eta,D}(h)}(x))  
= - g'(h(x))\E_x\left(  \int_0^{{\tau^{\eta} \wedge \tau_h}}A_X h(X_t)dt   \right) + o(\E_x(\tau_h)).
\end{align}
Thus,
\[\lim_{h\searrow 0}\frac{g(h(x)) - g(\psi_{\tau^{\lambda,C}\diamond \tau^{\eta,D}(h)}(x))}{\E_x(\tau_h)} = -g'(h(x))A_Xh(x).\] 
The result follows.
\end{proof}

\begin{lem}\label{smooth-fit-lemma}  For any $\tau^{\lambda,C},\tau^{\eta,D} \in \mathcal{N}$ and $x \in \partial C \cap D$,
\begin{align}
& \liminf_{h\searrow 0}\frac{J_{ {\lambda,C}}(x)-J_{\tau^{\lambda,C}\diamond \tau^{\eta,D}(h)}(x)}{\E_x(\tau_h)} \\
& \enskip= 
\liminf_{h\searrow 0}\frac{
\phi_{ {\lambda,C}}(x)-\E_x\left(\phi_{ {\lambda,C}}(X_{\tau_h})\right)
+g(\psi_{ {\lambda,C}}(x))- g\left( \E_x\left(\psi_{ {\lambda,C}}(X_{\tau_h})\right)\right) 
}{\E_x(\tau_h)}.
\label{smooth-fit-lemma:1}
\end{align}
\end{lem}

\begin{proof} 
Here we use the temporary notation $(A),(B)$ etc defined below. Write
\begin{align}
&J_{ {\lambda,C}}(x)-J_{\tau^{\lambda,C}\diamond \tau^{\eta,D}(h)}(x)\\
&= J_{ {\lambda,C}}(x)-J_{{\tau^{\lambda,C}}\circ \theta_{\tau_h}+\tau_h}(x)-(J_{\tau^{\lambda,C}\diamond \tau^{\eta,D}(h)}(x)-J_{{\tau^{\lambda,C}}\circ \theta_{\tau_h}+\tau_h}(x))\\
&= (A)-(B).
\end{align}
Write,
\begin{align}
&(B):= J_{\tau^{\lambda,C}\diamond \tau^{\eta,D}(h)}(x)-J_{{\tau^{\lambda,C}}\circ \theta_{\tau_h}+\tau_h}(x)\\
&\enskip =\phi_{\tau^{\lambda,C}\diamond \tau^{\eta,D}(h)}(x)-\phi_{{\tau^{\lambda,C}}\circ \theta_{\tau_h}+\tau_h}(x) + 
g(\psi_{\tau^{\lambda,C}\diamond \tau^{\eta,D}(h)}(x))-g(\psi_{{\tau^{\lambda,C}}\circ \theta_{\tau_h}+\tau_h}(x))\\
&\enskip =(B1) + (B2).
\end{align}
Use that $x\in D$ and the same arguments as for \eqref{proof-ref} to see that there exists a constant $\bar{h}>0$ such that, for each $0< h\leq\bar{h}$,
\begin{align}
 (B1)&:= \phi_{\tau^{\lambda,C}\diamond \tau^{\eta,D}(h)}(x)-\phi_{{\tau^{\lambda,C}}\circ \theta_{\tau_h}+\tau_h}(x)\\
& = \E_x\left(\int_0^{\tau_h}\eta(X_t)e^{-\int_0^t\eta(X_s)ds}\left(f(X_t)-\phi_{{\lambda,C}}\left(X_{\tau_h}\right)
\right)dt\right).
\end{align}
Similarly, using Taylor expansion, we obtain
\begin{align}
& (B2):= g(\psi_{\tau^{\lambda,C}\diamond \tau^{\eta,D}(h)}(x))-g(\psi_{{\tau^{\lambda,C}}\circ \theta_{\tau_h}+\tau_h}(x))\\
& \enskip = g\left( 
\psi_{{\tau^{\lambda,C}}\circ \theta_{\tau_h}+\tau_h}(x)
+ \E_x\left(\int_0^{\tau_h}\eta(X_t)e^{-\int_0^t\eta(X_s)ds}\left(h(X_t)-\psi_{ {\lambda,C}}\left(X_{\tau_h}\right)
\right)dt\right)	
\right)\\
& \enskip \enskip\enskip  - g\left(\psi_{{\tau^{\lambda,C}}\circ \theta_{\tau_h}+\tau_h}(x)\right)\\
&\enskip  = g'(\psi_{{\tau^{\lambda,C}}\circ \theta_{\tau_h}+\tau_h}(x))\E_x\left(\int_0^{\tau_h}\eta(X_t)e^{-\int_0^t\eta(X_s)ds}\left(h(X_t)-\psi_{ {\lambda,C}}\left(X_{\tau_h}\right)
\right)dt\right)\\
&\enskip \enskip\enskip  
+ o(\E_x(\tau_h))\\
& \enskip = g'(   \E_x\left(\psi_{\lambda,C}(X_{\tau_h})\right))\E_x\left(\int_0^{\tau_h}\eta(X_t)e^{-\int_0^t\eta(X_s)ds}\left(h(X_t)-\psi_{ {\lambda,C}}\left(X_{\tau_h}\right)
\right)dt\right)\\
&\enskip \enskip \enskip + o(\E_x(\tau_h)).
\end{align}
Since $\phi_{ {\lambda,C}}(x)-f(x)=0$ and $\psi_{ {\lambda,C}}(x)-h(x)=0$ for $x \in \partial C$, and these functions are continuous (cf. admissibility), it follows that
\begin{align}
\liminf_{h\searrow 0}\frac{-(B)}{\E_x(\tau_h)} = \liminf_{h\searrow 0}\frac{-(B1)-(B2)}{\E_x(\tau_h)} = 0.
\end{align} 
Write
\begin{align} 
(A):&= J_{ {\lambda,C}}(x)-J_{{\tau^{\lambda,C}}\circ \theta_{\tau_h}+\tau_h}(x) \\
& = 
\phi_{ {\lambda,C}}(x) + g(\psi_{ {\lambda,C}}(x)) 
-
\left(\phi_{{\tau^{\lambda,C}}\circ \theta_{\tau_h}+\tau_h}(x) 
+g(\psi_{{\tau^{\lambda,C}}\circ \theta_{\tau_h}+\tau_h}(x)) 
\right)\\
& = 
\phi_{ {\lambda,C}}(x) + g(\psi_{ {\lambda,C}}(x))
-
\left(\E_x\left(\phi_{ {\lambda,C}}(X_{\tau_h})\right)
+ g(\E_x\left(\psi_{ {\lambda,C}}(X_{\tau_h})\right))\right).
\end{align}
The result follows.
\end{proof}

\begin{proof} (of Theorem \ref{main-thm-imm-stp}). In this proof we use the notation $\hat\tau= \tau^{\lambda,C}$. 
Let us first suppose that $\hat\tau$ is an equilibrium stopping time, i.e. that it satisfies \eqref{eqdef2} for each $x \in E$ and each $\tau^{\eta,D} \in \mathcal{N}$, and show that this implies that conditions \eqref{mainthmcond1}--\eqref{mainthmcond5} are satisfied. Let us consider different cases for $x$.

\begin{itemize} 
\item  $x\in C$: Set  $D=C$ and use Lemma \ref{prop-calc-eq-cond1} to see that \eqref{eqdef2} can in this case be written as 
\[(\lambda(x)-\eta(x))\{f(x)-\phi_{\lambda,C}(x)+ g'(\psi_{\lambda,C}(x))\left(h(x)-\psi_{\lambda,C}(x)\right)\}\geq 0.\]
It follows that conditions \eqref{mainthmcond3} and \eqref{mainthmcond4} are satisfied. To see this recall that the non-negative function $\eta$ can be chosen so that $\eta(x)$ is arbitrarily large or $\eta(x)=0$.

Now set $D=\emptyset$, which implies that the numerator of the left side of \eqref{eqdef2} is  $J_{{\lambda,C}}(x) - f(x) - g(h(x))$, which does not depend on the constant $h$. This implies that  \eqref{mainthmcond1} holds. 

\item  $x\in \mbox{int}(C^c)$: Set  $D=\mbox{int}(C^c)$ and use Lemma \ref{prop-calc-eq-cond2} to see that \eqref{eqdef2} can in this case be written as $-A_Xf(x) - g'(h(x))A_X h(x)\geq 0$. It follows that condition \eqref{mainthmcond2} is satisfied. 

\item $x\in \partial C$: Set $D=E$ and use Lemma \ref{smooth-fit-lemma} to see that the left side of \eqref{eqdef2} is equal to the left side of the inequality in \eqref{mainthmcond5}, which directly implies that condition \eqref{mainthmcond5} holds.
\end{itemize}

Le us now suppose that $\hat\tau$ solves the system \eqref{mainthmcond1}--\eqref{mainthmcond5} and show that this implies that $\hat\tau$ is an equilibrium stopping time, i.e. that it satisfies \eqref{eqdef2} for each $x \in E$ and each $\tau^{\eta,D}\in \mathcal{N}$. Let us consider an arbitrary $\tau^{\eta,D}\in \mathcal{N}$ and different cases for $x$.
\begin{itemize} 

\item $x\in D$: 
\begin{itemize} 
\item  If $x \in C$ and $\lambda(x)>0$, then the left side of \eqref{eqdef2} is, by Lemma \ref{prop-calc-eq-cond1}, equal to $(\lambda(x)-\eta(x))\{f(x)-\phi_{\lambda,C}(x)+ g'(\psi_{\lambda,C}(x))\left(h(x)-\psi_{\lambda,C}(x)\right)\}$ and hence \eqref{mainthmcond3} implies that \eqref{eqdef2} must hold. 

\item  If $x \in C$ and $\lambda(x)=0$, then the left side of \eqref{eqdef2} is, by Lemma \ref{prop-calc-eq-cond1}, equal to $-\eta(x)\{f(x)-\phi_{\lambda,C}(x)+ g'(\psi_{\lambda,C}(x))\left(h(x)-\psi_{\lambda,C}(x)\right)\}$ and hence \eqref{mainthmcond4} implies that \eqref{eqdef2} must hold. 

\item If $x \in \mbox{int}(C^c)$, then Lemma \ref{prop-calc-eq-cond2} implies that the left side of \eqref{eqdef2} is equal to 
$-A_Xf(x) - g'(h(x))A_X h(x)$ and hence \eqref{mainthmcond2} implies that  \eqref{eqdef2} must hold. 

\item If $x \in \partial C$, then Lemma \ref{smooth-fit-lemma} and  \eqref{mainthmcond5} implies that  \eqref{eqdef2} must hold. 

\end{itemize}
\item $x\in D^c$: The numerator of the left side of \eqref{eqdef2} is in this case $J_{{\lambda,C}}(x) - f(x) - g(h(x))$ and hence \eqref{mainthmcond1}  implies that  \eqref{eqdef2} holds for $x\in C$. In case $x\notin C$ then the numerator is zero. 
\end{itemize}
\end{proof}

\begin{lem}\label{lemma-tau-h-oct}  For any $x \in \state$ holds, 
\begin{align}
\lim_{h\searrow 0} \frac{h^2}{\E_x(\tau_h)} = \sigma^2(x) 
\enskip \mbox{ and } \enskip 
\lim_{h\searrow 0} \frac{\E_x(\tau_h^2)}{\E_x(\tau_h)} = 0. 
\end{align}
\end{lem}
\begin{proof} Consider a fixed $x \in \state$, a constant $a>0$ and let 
$F(t,y):=a(y-x)^2-t$. 
Using simple calculations we find that there exists a constant $\bar{h}$ such that for $y\in (x-\bar{h},x+\bar{h})$ holds
$\left(\frac{\partial}{\partial t} + A_X\right)F(t,y) =  - 1 + a\sigma^2(y) + \mu(y)2a(y-x)\geq 0$ whenever $a\sigma^2(x)>1$; 
which with It\^{o}'s formula and the optional sampling theorem gives 
$\E_x(F(\tau_h,X_{\tau_h})) = ah^2-\E_x(\tau_h)\geq 0$ for $h<\bar{h}$. 
Hence, 
$\liminf_{h\searrow 0} \frac{h^2}{\E_x(\tau_h)} \geq \frac{1}{a}$. 
With the same arguments we find that 
$\limsup_{h\searrow 0} \frac{h^2}{\E_x(\tau_h)} \leq \frac{1}{b}$ for $b\sigma^2(x)<1$ 
and the first claim of the lemma follows directly. 

Now let $F(t,y):=at(y-x)^2-t^2$. Then there exists a constant $\bar{h}$ such that for $y\in (x-\bar{h},x+\bar{h})$ holds 
$\left(\frac{\partial}{\partial t} + A_X\right)F(t,y) =   a(y-x)^2 + t (a (2\mu(y)(y-x) +\sigma^2(y))-2)>0$, for $a\sigma^2(x)>2$. Using the same arguments as above we can now show that
$\E_x(a\tau_h h^2-\tau_h^2)\geq 0$ for $h<\bar{h}$. The second claim\deleted[id=kri,remark={}]{s} follows. 

\end{proof}

\bibliographystyle{abbrv}
\bibliography{time-incon_stopping}

\begin{thebibliography}{10}

\bibitem{bayraktar2018}
E.~Bayraktar, J.~Zhang, and Z.~Zhou.
\newblock Time consistent stopping for the mean-standard deviation
  problem---the discrete time case.
\newblock {\em SIAM Journal on Financial Mathematics}, 10(3):667--697, 2019.

\bibitem{bensoussan2014time}
A.~Bensoussan, K.~Wong, S.~C.~P. Yam, and S.-P. Yung.
\newblock Time-consistent portfolio selection under short-selling prohibition:
  From discrete to continuous setting.
\newblock {\em SIAM Journal on Financial Mathematics}, 5(1):153--190, 2014.

\bibitem{bielecki2013credit}
T.~R. Bielecki and M.~Rutkowski.
\newblock {\em Credit risk: modeling, valuation and hedging}.
\newblock Springer Science \& Business Media, 2013.

\bibitem{tomas-continFORTH}
T.~Bj{\"o}rk, M.~Khapko, and A.~Murgoci.
\newblock On time-inconsistent stochastic control in continuous time.
\newblock {\em Finance and Stochastics}, 21(2):331--360, 2017.

\bibitem{tomas-disc}
T.~Bj{\"o}rk and A.~Murgoci.
\newblock A theory of {M}arkovian time-inconsistent stochastic control in
  discrete time.
\newblock {\em Finance and Stochastics}, 18(3):545--592, 2014.

\bibitem{bjork2014mean}
T.~Bj{\"o}rk, A.~Murgoci, and X.~Y. Zhou.
\newblock Mean--variance portfolio optimization with state-dependent risk
  aversion.
\newblock {\em Mathematical Finance}, 24(1):1--24, 2014.

\bibitem{buonaguidi2015remark}
B.~Buonaguidi.
\newblock A remark on optimal variance stopping problems.
\newblock {\em Journal of Applied Probability}, 52(4):1187--1194, 2015.

\bibitem{buonaguidi2018some}
B.~Buonaguidi and A.~Mira.
\newblock Some optimal variance stopping problems revisited with an application
  to the {Italian} {Ftse-Mib} stock index.
\newblock {\em Sequential Analysis}, 37(1):90--101, 2018.

\bibitem{christensen2017finding}
S.~Christensen and K.~Lindensj{\"o}.
\newblock On finding equilibrium stopping times for time-inconsistent
  {M}arkovian problems.
\newblock {\em SIAM Journal on Control and Optimization}, 56(6):4228--4255,
  2018.

\bibitem{dupuis2002optimal}
P.~Dupuis and H.~Wang.
\newblock Optimal stopping with random intervention times.
\newblock {\em Advances in Applied probability}, 34(1):141--157, 2002.

\bibitem{Duraj2017Optimal}
J.~Duraj.
\newblock Optimal stopping with general risk preferences.
\newblock {\em SSRN preprint:2897765}, 2017.

\bibitem{ebert2017discounting}
S.~Ebert, W.~Wei, and X.~Y. Zhou.
\newblock Discounting, diversity, and investment.
\newblock {\em SSRN preprint 2840240}, 2017.

\bibitem{gad2016optimal}
K.~S.~T. Gad and P.~Matom{\"a}ki.
\newblock Optimal variance stopping with linear diffusions.
\newblock {\em Stochastic Processes and their Applications}, 2019.

\bibitem{gad2015variance}
K.~S.~T. Gad and J.~L. Pedersen.
\newblock Variance optimal stopping for geometric {L{\'e}vy} processes.
\newblock {\em Advances in Applied Probability}, 47(1):128--145, 2015.

\bibitem{geiss2017first}
C.~Geiss, A.~Luoto, and P.~Salminen.
\newblock On first exit times and their means for {B}rownian bridges.
\newblock {\em Journal of Applied Probability}, 56(3):701--722, 2019.

\bibitem{guo2005stopping}
X.~Guo and J.~Liu.
\newblock Stopping at the maximum of geometric {B}rownian motion when signals
  are received.
\newblock {\em Journal of Applied Probability}, 42(3):826--838, 2005.

\bibitem{he2013optimal}
L.~He and Z.~Liang.
\newblock Optimal investment strategy for the {DC} plan with the return of
  premiums clauses in a mean--variance framework.
\newblock {\em Insurance: Mathematics and Economics}, 53(3):643--649, 2013.

\bibitem{he2018dynamic}
X.~D. He and Z.~Jiang.
\newblock Dynamic mean-risk asset allocation.
\newblock {\em Available at SSRN 3084657}, 2018.

\bibitem{he2019Equilibrium}
X.~D. He and Z.~Jiang.
\newblock On the equilibrium strategies for time-inconsistent problems in
  continuous time.
\newblock {\em Available at SSRN 3308274}, 2018.

\bibitem{huang2018time}
Y.-J. Huang and A.~Nguyen-Huu.
\newblock Time-consistent stopping under decreasing impatience.
\newblock {\em Finance and Stochastics}, 22(1):69--95, 2018.

\bibitem{huang2017stopping}
Y.-J. Huang, A.~Nguyen-Huu, and X.~Y. Zhou.
\newblock General stopping behaviors of na{\"\i}ve and noncommitted
  sophisticated agents, with application to probability distortion.
\newblock {\em Mathematical Finance}, 30(1):310--340, 2020.

\bibitem{huang2017optimal}
Y.-J. Huang and Z.~Zhou.
\newblock Optimal equilibria for time-inconsistent stopping problems in
  continuous time.
\newblock {\em arXiv preprint arXiv:1712.07806}, 2018.

\bibitem{huang2018strong}
Y.-J. Huang and Z.~Zhou.
\newblock Strong and weak equilibria for time-inconsistent stochastic control
  in continuous time.
\newblock {\em arXiv preprint arXiv:1809.09243}, 2018.

\bibitem{huang2017optimalDISC}
Y.-J. Huang and Z.~Zhou.
\newblock The optimal equilibrium for time-inconsistent stopping problems---the
  discrete-time case.
\newblock {\em SIAM Journal on Control and Optimization}, 57(1):590--609, 2019.

\bibitem{karatzas2012brownian}
I.~Karatzas and S.~Shreve.
\newblock {\em Brownian motion and stochastic calculus}, volume 113.
\newblock Springer Science \& Business Media, 2012.

\bibitem{kronborg2015inconsistent}
M.~T. Kronborg and M.~Steffensen.
\newblock Inconsistent investment and consumption problems.
\newblock {\em Applied Mathematics \& Optimization}, 71(3):473--515, 2015.

\bibitem{lando1998cox}
D.~Lando.
\newblock On cox processes and credit risky securities.
\newblock {\em Review of Derivatives research}, 2(2-3):99--120, 1998.

\bibitem{li2013optimal}
Y.~Li and Z.~Li.
\newblock Optimal time-consistent investment and reinsurance strategies for
  mean--variance insurers with state dependent risk aversion.
\newblock {\em Insurance: Mathematics and Economics}, 53(1):86--97, 2013.

\bibitem{lindensjo2017timeinconHJB}
K.~Lindensj{\"o}.
\newblock A regular equilibrium solves the extended {HJB} system.
\newblock {\em Operations Research letters}, 47(5):427--432, 2019.

\bibitem{markowitz1952portfolio}
H.~Markowitz.
\newblock Portfolio selection.
\newblock {\em The journal of finance}, 7(1):77--91, 1952.

\bibitem{miller2017nonlinear}
C.~W. Miller.
\newblock Nonlinear {PDE} approach to time-inconsistent optimal stopping.
\newblock {\em SIAM Journal on Control and Optimization}, 55(1):557--573, 2017.

\bibitem{nutz2018mean}
M.~Nutz.
\newblock A mean field game of optimal stopping.
\newblock {\em SIAM Journal on Control and Optimization}, 56(2):1206--1221,
  2018.

\bibitem{pedersen2011explicit}
J.~L. Pedersen.
\newblock Explicit solutions to some optimal variance stopping problems.
\newblock {\em Stochastics An International Journal of Probability and
  Stochastic Processes}, 83(4-6):505--518, 2011.

\bibitem{pedersen2016optimal}
J.~L. Pedersen and G.~Peskir.
\newblock Optimal mean--variance selling strategies.
\newblock {\em Mathematics and Financial Economics}, 10(2):203--220, 2016.

\bibitem{peskir2005change}
G.~Peskir.
\newblock A change-of-variable formula with local time on curves.
\newblock {\em Journal of Theoretical Probability}, 18(3):499--535, 2005.

\bibitem{peskir2006optimal}
G.~Peskir and A.~Shiryaev.
\newblock {\em Optimal stopping and free-boundary problems}.
\newblock Springer, 2006.

\bibitem{selten1965spieltheoretische}
R.~Selten.
\newblock Spieltheoretische behandlung eines oligopolmodells mit
  nachfragetr{\"a}gheit: Teil i: Bestimmung des dynamischen
  preisgleichgewichts.
\newblock {\em Zeitschrift f{\"u}r die gesamte Staatswissenschaft/Journal of
  Institutional and Theoretical Economics}, (H. 2):301--324, 1965.

\bibitem{selten1975reexamination}
R.~Selten.
\newblock Reexamination of the perfectness concept for equilibrium points in
  extensive games.
\newblock {\em International journal of game theory}, 4(1):25--55, 1975.

\bibitem{strotz}
R.~Strotz.
\newblock Myopia and inconsistency in dynamic utility maximization.
\newblock {\em The Review of Economic Studies}, 23(3):165--180, 1955.

\bibitem{touzi2002continuous}
N.~Touzi and N.~Vieille.
\newblock Continuous-time {D}ynkin games with mixed strategies.
\newblock {\em SIAM Journal on Control and Optimization}, 41(4):1073--1088,
  2002.

\end{thebibliography}

\end{document}